\documentclass[11pt,reqno]{amsart}
   \usepackage[centertags]{amsmath}
   \usepackage{amsfonts}
   \usepackage{amssymb,mathtools}
   \usepackage{amsthm}
   \usepackage[utf8]{inputenc}
   \usepackage{enumitem,xcolor}
\usepackage{array}
\usepackage{longtable}

\usepackage{caption}
\usepackage{graphicx}
\usepackage{subcaption}

\usepackage{hyperref}

\usepackage{multirow, bigdelim}

\usepackage{tikz}

\usetikzlibrary{arrows.meta, positioning}

\usepackage{geometry}

\usepackage{ifmtarg}

\usepackage{color}

\usepackage{bm}

\newcommand\D{\displaystyle}

\makeatletter

    \newcommand\contFrac{\@ifstar{\@contFracStar}{\@contFracNoStar}}

    \def\singleContFrac#1#2{%
        \begin{array}{@{}c@{}}%
            \multicolumn{1}{c|}{#1}%
            \\%
            \hline%
            \multicolumn{1}{|c}{#2}%
        \end{array}%
    }

    \def\@contFracNoStar#1{%
        \mathchoice{% * Display style
            \@contFracNoStarDisplay@#1//\@nil%
        }{%           * Text style
            \@contFracNoStarInline@#1//\@nil%
        }{%           * Script style
            \@contFracNoStarInline@#1//\@nil%
        }{%           * Script script style
            \@contFracNoStarInline@#1//\@nil%
        }%
    }

    \def\@contFracNoStarDisplay@#1//#2\@nil{%
        \@ifmtarg{#2}{%
            #1%
        }{%
            #1+\cfrac{1}{\@contFracNoStarDisplay@#2\@nil}%
        }%
    }

        \def\@contFracNoStarInline@#1//#2\@nil{%
            \@ifmtarg{#2}{%
                #1%
            }{%
                #1 \@@contFracNoStarInline@@#2\@nil%
            }%
        }
        \def\@@contFracNoStarInline@@#1//#2\@nil{%
            \@ifmtarg{#2}{%
                + \singleContFrac{1}{#1}%
            }{%
                + \singleContFrac{1}{#1} \@@contFracNoStarInline@@#2\@nil%
            }%
        }

    \def\@contFracStar#1{%
        \mathchoice{% * Display style
            \@contFracStarDisplay@#1////\@nil%
        }{%           * Text style
            \@contFracStarInline@#1//\@nil%
        }{%           * Script style
            \@contFracStarInline@#1//\@nil%
        }{%           * Script script style
            \@contFracStarInline@#1//\@nil%
        }%
    }

    \def\@contFracStarDisplay@#1//#2//#3\@nil{%
        \@ifmtarg{#2}{%
            #1%
        }{%
            #1 + \cfrac{#2}{\@contFracStarDisplay@#3\@nil}%
        }%
    }

        \def\@contFracStarInline@#1//#2\@nil{%
            \@ifmtarg{#2}{%
                #1%
            }{%
                #1 \@@contFracStarInline@@#2\@nil%
            }%
        }
        \def\@@contFracStarInline@@#1//#2//#3\@nil{%
            \@ifmtarg{#3}{%
                - \singleContFrac{#1}{#2}%
            }{%
                - \singleContFrac{#1}{#2} \@@contFracStarInline@@#3\@nil%
            }%
        }
\makeatother

\allowdisplaybreaks[3]

\theoremstyle{plain}
\newtheorem{theorem}{Theorem}[section]
\newtheorem{lemma}[theorem]{Lemma}
\newtheorem{proposition}[theorem]{Proposition}

\theoremstyle{definition}

\theoremstyle{remark}
\newtheorem{remark}[theorem]{Remark}
\newtheorem*{remark*}{Remark}

\numberwithin{equation}{section}

\title[]{LU-type factorizations and Darboux transformations\\ for continuous-time birth--death processes}

\author{Jos\'e Arcia-Manoleskos}
\address{Jos\'e Arcia-Manoleskos, Departamento de An\'{a}lisis Matem\'{a}tico. Universidad de Sevilla. Apdo (P. O. BOX) 1160. 41080 Sevilla. Spain.}
\address{Departamento de Matem\'{a}tica. Facultad de Ciencias Naturales, Exactas y Tecnolog\'{i}a. Universidad de Panam\'a}
\email{josarcman@alum.us.es}
\email{jose.arciam@up.ac.pa}

\author{Manuel Dom\'inguez de la Iglesia}
\address{Manuel Dom\'inguez de la Iglesia, Universidad de Alcal\'a, Departamento de F\'isica y Matem\'aticas, Campus universitario, E-28805 Alcal\'a de Henares (Madrid), Spain.}
\email{manuel.dominguezi@uah.es}

\begin{document}

\begin{abstract}
We study LU-type factorizations of the infinitesimal generator of a possibly nonconservative birth--death process on $\mathbb{N}_0$, and characterize those for which the corresponding Darboux transformations, obtained by reversing the order of the factors, are again infinitesimal generators of birth--death processes. Two types of factorizations are considered: lower--upper (LU) and upper--lower (UL). In the LU case, the factorization is uniquely determined once the absorption rate of the Darboux-transformed process is fixed, whereas the UL factorization involves an additional free parameter. For both cases we obtain closed-form expressions for the coefficients of the factors in terms of the birth--death rates, determine the exact ranges of the parameters for which the Darboux transforms are birth--death generators, and give probabilistic interpretations of the coefficients in terms of hitting probabilities, expected passage times and occupation times. We also describe the associated spectral transformations (Christoffel and Geronimus transforms) and orthogonal polynomials, and study when the Darboux-transformed process again has a uniquely determined spectral measure: this is always the case for the UL factorization, while for the LU factorization it is characterized by the divergence of a single explicit series. The general results are illustrated with the $M/M/1$ queue, for which all the quantities involved can be computed explicitly.

\smallskip
\noindent\textbf{Keywords:} birth--death processes, LU factorizations, Darboux transformations, orthogonal polynomials, Christoffel and Geronimus transformations, spectral measures.

\smallskip
\noindent\textbf{2020 Mathematics Subject Classification:} 60J27, 33C45, 42C05.
\end{abstract}

\maketitle

\section{Introduction}

Birth--death processes form a fundamental subclass of continuous-time Markov chains,
characterized by the fact that transitions can occur only between neighboring states.
Such processes naturally arise in a wide range of scientific disciplines, including
biology, genetics, ecology, physics, mathematical finance, queueing and communication
systems, epidemiology, and chemical kinetics. For a comprehensive overview, see
\cite{PL}.

\smallskip

The motivation for this work comes from \cite{GdI}, where the authors examine
\emph{stochastic} LU-type factorizations of the transition probability matrix of
discrete-time birth--death chains on $\mathbb{N}_0$ (also commonly referred to as random
walks; see, for instance, \cite{KMc6}). In that framework, since the product of two
stochastic matrices is again stochastic, the corresponding \emph{Darboux
transformation}, obtained by reversing the order of the factors, also yields a
stochastic matrix. This provides a systematic method for constructing families of
stochastic matrices with closely related spectral measures, and it also offers a
mechanism for simplifying the urn models associated with these processes. This approach
has been successfully extended to other settings, including birth--death chains on
$\mathbb{Z}$ \cite{dIJ1,dIJ2,dIJ4}, QBD processes \cite{GdI4}, and birth--death chains
on a spider \cite{dIJ3}, all within the discrete-time framework. LU-type factorizations
and Darboux transformations have also been extensively studied in the theory of
integrable systems, particularly in connection with Krall-type orthogonal polynomials;
see, for instance, \cite{GrH1,GrH3,GrHH}.

\smallskip

In this paper, we focus on continuous-time birth--death chains on $\mathbb{N}_0$,
commonly referred to as \emph{birth--death processes}, possibly with absorption at a
state $-1$ accessible from $0$. Our analysis starts from their infinitesimal generator
$\mathcal{A}$, which is a tridiagonal (Jacobi) matrix (see \eqref{QQmm} below). We
consider two types of LU-type factorizations of $\mathcal{A}$, namely a lower--upper
(LU) and an upper--lower (UL) decomposition,
\[
\mathcal{A} = \widetilde{\mathcal{A}}_L \widetilde{\mathcal{A}}_U
\quad\text{and}\quad
\mathcal{A} = \mathcal{A}_U \mathcal{A}_L,
\]
where $\mathcal{A}_L$ and $\widetilde{\mathcal{A}}_L$ are lower bidiagonal matrices, and
$\mathcal{A}_U$ and $\widetilde{\mathcal{A}}_U$ are upper bidiagonal matrices. In
principle, many such factorizations exist, since they depend on free parameters. Our
focus is on those whose associated Darboux transformations,
\begin{equation*}
\widehat{\mathcal{A}} = \widetilde{\mathcal{A}}_U \widetilde{\mathcal{A}}_L
\quad \text{and} \quad
\widetilde{\mathcal{A}} = \mathcal{A}_L \mathcal{A}_U,
\end{equation*}
are \emph{again} infinitesimal generators of birth--death processes.

\smallskip

LU-type factorizations of stochastic matrices have been studied by several authors
\cite{Gr1,Gr2,hey,Vig}, mainly because they provide effective tools for computing
invariant measures and, more generally, for analyzing structured Markov chains
\cite{Vig}. In particular, Heyman \cite{hey} obtained a factorization of the form
\(I-P=(U-I)(V-D)\), where \(P\) is the transition probability matrix of an ergodic
discrete-time Markov chain, \(U\) is strictly upper triangular, \(V\) is strictly lower
triangular, and \(D\) is diagonal. The entries of \(U\), interpreted as expected values,
and those of \(V\) and \(D\), interpreted as probabilities, are associated with so-called
\emph{censored Markov chains}. On the other hand, \cite{Vig} considers factorizations of
the form \(I-P=(I-L)(I-K)\), where \(L\) and \(K\) are upper and lower triangular
matrices, respectively, both with nonnegative entries; in that setting, each factor can
be interpreted as the generator of an increasing or decreasing Markov chain. In contrast
with \cite{hey,Vig}, our analysis is carried out directly on infinitesimal generators
rather than on transition probability matrices, and the factors are not, in general,
generators of Markov processes. The key feature of the present approach is that the
Darboux transformation is itself the generator of a new birth--death process, an aspect
that was not explored in \cite{hey,Vig}.

\smallskip

One might expect a direct extension of the discrete-time results of \cite{GdI} to the
continuous-time setting. However, several features make the continuous-time case
considerably more intricate:

\begin{itemize}[leftmargin=0.25in]
\item In discrete time, the factors were stochastic matrices and their entries were
transition probabilities. In continuous time, the factors are not, in general, related
to Markov processes, and the requirement that the Darboux transformation is again a
birth--death generator is no longer automatic: it imposes explicit restrictions on the
parameters of the factorization. Nevertheless, the entries of the factors still admit a
probabilistic meaning, now in terms of hitting probabilities, expected passage times and
occupation times of certain stopping times, as we will see in Sections~\ref{Sec3} and~\ref{Sec4}.
\smallskip

\item While in the discrete-time setting all matrices involved were purely stochastic (nonabsorbing or conservative birth--death chains), here we allow the birth--death process to include an absorbing state (typically denoted by $-1$), accessible from state $0$, leading to nonconservative birth--death processes. This generalization introduces additional complexity, as it requires the inclusion of extra parameters representing the absorption rates of both the original and the Darboux-transformed processes.

\smallskip

\item In discrete time, a stochastic Darboux transformation modifies the spectral
measure only at $x=0$, so recurrence, which is governed by the behavior of the spectral
measure at $x=1$, is preserved. In continuous time, recurrence (or the certainty of
absorption in the nonconservative case) is governed by the behavior of the spectral
measure near $x=0$, precisely where the Darboux transformation acts. As a consequence, a
transient process may become positive recurrent after a Darboux transformation, and vice
versa, as we will see, for instance, for the $M/M/1$ queue in Section~\ref{secMM1}.
\end{itemize}

\smallskip

An important advantage of working with tridiagonal operators such as $\mathcal{A}$ is
the availability of a well-developed spectral theory, involving a spectral measure and an
associated family of orthogonal polynomials. This approach, originally developed by
Karlin and McGregor \cite{KMc2,KMc3}, provides an integral representation of the
transition probabilities, an explicit expression of the invariant measure (when it
exists) in terms of the squared norms of the polynomials, and spectral characterizations
of recurrence, absorption, first-passage times and limit theorems. For a comprehensive
recent account, see~\cite{MDIB} (see also \cite{Scho}). In our setting, the spectral
measures of the Darboux-transformed processes are obtained from that of the original
process by a \emph{Christoffel transformation} (LU case) or a \emph{Geronimus transformation} (UL case), so that all these tools remain available for the new processes, provided their
spectral measure is again uniquely determined by the rates.

\smallskip

The main contributions of the paper are the following:

\begin{itemize}[leftmargin=0.25in]
\item We obtain closed-form expressions for all the coefficients of the LU and UL
factorizations in terms of the birth--death rates, improving on the discrete-time case
\cite{GdI}, where they were only given recursively.
\smallskip

\item We characterize exactly the admissible parameters for which the Darboux transforms
are birth--death generators (Theorems~\ref{theom33} and~\ref{theom37}). In the LU case,
the factorization is uniquely determined by the absorption rate $\hat\mu_0$ of the
transformed process, which must lie in an explicit interval. In the UL case, there is an
additional free parameter $x_0$, and two different regimes appear, depending on whether
$x_0+\tilde\mu_0$ attains its maximal admissible value or not.
\smallskip

\item We give probabilistic interpretations of the coefficients of the factors in terms
of hitting probabilities, conditional expected passage times and occupation times,
together with their connections with Doob $h$-transforms and dual processes.
\smallskip

\item We study whether the Darboux-transformed process again satisfies the standing condition ensuring that its spectral measure is uniquely determined. This is always the case for the UL factorization (Proposition~\ref{propSUL}), whereas for the LU factorization it holds if and only if a single explicit series, depending only on the original rates and not on $\hat\mu_0$, diverges (see Proposition~\ref{propSLU}).
\end{itemize}

\smallskip

Throughout the paper, the birth--death process is understood as the minimal process
associated with its rates. This is no restriction for our results, which only involve
the behavior of the process before it leaves every finite set of states; possible
explosion is briefly discussed in Remark~\ref{remexplosion}.

\smallskip

The structure of the paper is as follows. Section~\ref{sec2} collects some preliminaries
on birth--death processes, their spectral representation, and some hitting
probabilities, expected passage times and occupation times that will be used later.
Sections~\ref{Sec3} and~\ref{Sec4} contain the main results. Section~\ref{Sec3} is
devoted to the LU factorization of $\mathcal{A}$: the explicit formulas for its
coefficients, the characterization of the admissible parameters (Theorem~\ref{theom33}),
the probabilistic interpretation of the coefficients, the associated spectral measure and
orthogonal polynomials, and the study of the standing condition for the
Darboux-transformed process (Proposition~\ref{propSLU}). Section~\ref{Sec4} contains the
analogous results for the UL factorization of $\mathcal{A}$ (Theorem~\ref{theom37} and
Proposition~\ref{propSUL}). The general theory is
illustrated in Section~\ref{secMM1} with the $M/M/1$ queue. For this example, the spectral
measure and orthogonal polynomials of the original process are classical (Chebyshev), and
all the quantities involved in the factorizations, including those of the
Darboux-transformed processes, can be computed explicitly. Finally, Section~\ref{secConc}
contains some concluding remarks and open questions, including further examples, such as
the $M/M/\infty$ queue, linear birth--death processes and birth--death processes with
quadratic rates, which will be studied in a forthcoming work.

\section{Preliminaries}\label{sec2}

Let $\{X_t:t\geq0\}$ be a birth--death process on $\mathbb{N}_0$, possibly killed at the absorbing state $-1$, with birth and death rates $\{\lambda_n,\mu_n\}$ and infinitesimal generator $\mathcal{A}$, given by the semi-infinite matrix
\begin{equation}\label{QQmm}
\mathcal{A}=
	\begin{pmatrix}
		-(\lambda_0+\mu_0) & \lambda_0 &  &  & \\
		\mu_1  & -(\lambda_1+\mu_1)  & \lambda_1 &  &\\
		 & \mu_2 & -(\lambda_2+\mu_2) & \lambda_2 & \\
		&& \ddots & \ddots & \ddots
	\end{pmatrix}.
\end{equation}
The process is assumed to be irreducible on the nonabsorbing state space $\mathbb{N}_0$. A diagram of the transitions between states is
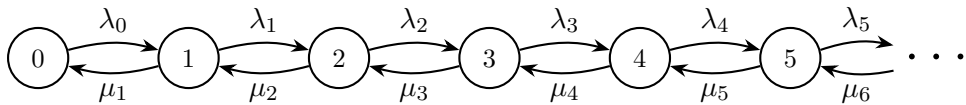
\begin{center}
\begin{tikzpicture}[
  node distance=2cm,
  state/.style={circle, draw, minimum size=8mm, inner sep=0pt, font=\small},
  >=Stealth, thick
]

% ==== States ====
\node[state] (m1) {$-1$};
\node[state] (0) [right of=m1] {0};
\node[state] (1) [right of=0] {1};
\node[state] (2) [right of=1] {2};
\node[state] (3) [right of=2] {3};
\node[state] (4) [right of=3] {4};
\node[state] (5) [right of=4] {5};
\node (dots) [right of=5] {\huge$\cdots$};

% ==== Absorption at -1 ====
\draw[->] (0) to[bend right=20] node[below] {$\mu_0$} (m1);

% ==== Forward transitions (lambdas) ====
\foreach \i/\j/\lab in {
  0/1/$\lambda_0$,
  1/2/$\lambda_1$,
  2/3/$\lambda_2$,
  3/4/$\lambda_3$,
  4/5/$\lambda_4$,
  5/dots/$\lambda_5$
}{
  \draw[->, bend left=15] (\i) to node[above] {\lab} (\j);
}

% ==== Backward transitions (mus) ====
\foreach \i/\j/\lab in {
  1/0/$\mu_1$,
  2/1/$\mu_2$,
  3/2/$\mu_3$,
  4/3/$\mu_4$,
  5/4/$\mu_5$,
  dots/5/$\mu_6$
}{
  \draw[->, bend left=15] (\i) to node[below] {\lab} (\j);
}

\end{tikzpicture}
\end{center}
$\mathcal{A}$ is conservative if and only if $\mu_0=0$. When $\mu_0>0$, the process may jump from state $0$ to the absorbing state $-1$. More precisely, conditional on leaving state $0$, the probability that the next jump is to $-1$ is $\mu_0/(\lambda_0+\mu_0)$, while the probability that the next jump is to state $1$ is $\lambda_0/(\lambda_0+\mu_0)$. In some cases, we will denote by $\mathcal{A}^{(\alpha)}$ the generator obtained from \(\mathcal A\) by replacing \(\mu_0\) with \(\alpha\), leaving all other rates unchanged. Similarly, the probability measure and the expected value associated with the birth--death process generated by $\mathcal{A}^{(\alpha)}$ will be denoted by $\mathbb{P}^{(\alpha)}$ and $\mathbb{E}^{(\alpha)}$, respectively. When no superscript is indicated, it is understood that $\mathcal{A} = \mathcal{A}^{(\mu_0)}$.

\smallskip

Define the \emph{potential coefficients} $(\pi_n)_{n\in\mathbb{N}_0}$ of the birth--death process as
\begin{equation}\label{potcoef}
\pi_0=1,\quad\pi_n=\frac{\lambda_0\lambda_1\cdots\lambda_{n-1}}{\mu_1\mu_2\cdots\mu_n},\quad n\geq1.
\end{equation}
These potential coefficients satisfy the symmetry conditions
\begin{equation}\label{sympotcoef}
\lambda_n\pi_n=\mu_{n+1}\pi_{n+1},\quad n\geq0.
\end{equation}
Additionally, let us introduce the quantities
\begin{equation}\label{quantity_AB}
A=\sum_{n=0}^{\infty}\frac{1}{\lambda_n\pi_n},
\quad
B=\sum_{n=0}^{\infty}\pi_n,
\end{equation}
which will play an important role throughout the paper, together with
\begin{equation}\label{SSS}
S=\sum_{n=0}^{\infty}
\pi_n
\left(
1+\mu_0\sum_{k=0}^{n-1}\frac{1}{\lambda_k\pi_k}
\right)^2.
\end{equation}
Throughout this paper, we assume that the birth--death rates
$\{\lambda_n,\mu_n\}$ satisfy
\begin{equation}\label{uniqueness_assumption}
\begin{cases}
A+B=\infty, & \mu_0=0,\\[1mm]
S=\infty, & \mu_0>0.
\end{cases}
\end{equation}
If $\mu_0=0$, the condition $A+B=\infty$ is equivalent to the determinacy of the
associated Stieltjes moment problem; see \cite[Theorems~14 and~15]{KMc2} and
\cite[Section~3.1]{vanDoorn2016}. If $\mu_0>0$, determinacy of the associated moment
problem is equivalent to the condition $S=\infty$; see
\cite{vanDoorn2003,vanDoorn2016}. Thus, under \eqref{uniqueness_assumption}, the
associated spectral measure is uniquely determined by the rates. Throughout the paper, $\{X_t:t\geq0\}$ denotes the minimal process associated with
$\mathcal{A}$, with transition probabilities
\[
P_{ij}(t)
=
\mathbb{P}_i(X_t=j)
:=
\mathbb{P}(X_t=j\mid X_0=i),
\quad
i,j\in\mathbb{N}_0,\quad t\geq0.
\]
The spectral representation of Karlin and McGregor \cite{KMc2} yields a unique
probability measure $d\psi$ supported on $[0,\infty)$ and a family of polynomials
$(Q_n)_{n\in\mathbb{N}_0}$ satisfying the three-term recurrence relation
$-xQ=\mathcal{A}Q$, where $Q=(Q_0,Q_1,\ldots)^T$, namely
\begin{equation}\label{QQs}
\begin{split}
Q_{-1}(x)&=0,\quad Q_0(x)=1,\\
-xQ_n(x)
&=
\lambda_nQ_{n+1}(x)
-(\lambda_n+\mu_n)Q_n(x)
+\mu_nQ_{n-1}(x),
\quad n\geq0.
\end{split}
\end{equation}
The transition probabilities of the minimal process are then given by the
\emph{Karlin--McGregor formula}
\begin{equation}\label{3FKMcGN}
P_{ij}(t)
=
\pi_j\int_0^\infty
e^{-xt}Q_i(x)Q_j(x)\,d\psi(x),
\quad i,j\in\mathbb{N}_0.
\end{equation}

\begin{remark}\label{remexplosion}
Condition \eqref{uniqueness_assumption} ensures that the spectral measure is uniquely
determined by the rates, but it does not exclude \emph{explosion}, that is, reaching infinity in finite time with positive probability. In that case the rates no longer determine a unique process, since it may be continued in different ways after explosion, and \eqref{3FKMcGN} describes the \emph{minimal} one, which is killed at the explosion time. Explosion is characterized by Reuter's criterion (see \cite{And, Reu}). It can only occur for transient processes with rapidly growing rates; for instance, for $\lambda_n=3\cdot2^n$ and $\mu_n=2^n$ the moment problem is determinate but the process explodes. We do not exclude this situation, since it plays no role in our results: all processes are understood as minimal processes, and all the probabilistic quantities used below (hitting probabilities between finite sets of states, conditional expected passage times and occupation times of truncated processes) only
depend on the evolution of the process before it leaves every finite set of states. Moreover, the example of Section~\ref{secMM1} is not explosive.
\end{remark}

Going back to the three-term recurrence relation satisfied by the polynomials, evaluation of \eqref{QQs} at $x=0$ (see \cite[Formula~(3.42)]{MDIB}) gives 
\begin{equation}\label{Qn0}
Q_n(0) = 1 + \mu_0 \sum_{k=0}^{n-1} \frac{1}{\lambda_k \pi_k}, 
\quad n \ge 0,
\end{equation}
so that the quantity in \eqref{SSS} can equivalently be written as
\[
S=\sum_{n=0}^{\infty}\pi_nQ_n^2(0).
\]
When \(\mu_0>0\), \eqref{Qn0} can equivalently be written as
\begin{equation*}
Q_n(0) =\mu_0\sum_{k=0}^n\frac{1}{\mu_k\pi_k}, 
\quad n \ge 0.
\end{equation*}
The three-term recurrence relation \eqref{QQs} also yields the
\emph{Christoffel--Darboux formula}
\begin{equation}\label{CDF}
\sum_{k=0}^n\pi_kQ_k(x)Q_k(y)
=
\lambda_n\pi_n
\frac{
Q_{n+1}(x)Q_n(y)-Q_n(x)Q_{n+1}(y)
}{y-x},
\quad n\geq0.
\end{equation}
We shall also use the sequence of \emph{$0$-th associated polynomials} $(Q_n^{(0)})_{n\in\mathbb{N}_0}$, which we shall simply call the associated polynomials; see \cite[Section~3.3.2]{MDIB}. They form a second linearly independent solution of the three-term recurrence relation \eqref{QQs} for $n\geq1$, determined by the initial conditions
\begin{equation}\label{iccsqn0}
Q_0^{(0)}(x)=0,
\quad
Q_1^{(0)}(x)=-\frac{1}{\lambda_0}.
\end{equation}
A corresponding Christoffel--Darboux identity relating the two families
$(Q_n)_{n\in\mathbb{N}_0}$ and
$(Q_n^{(0)})_{n\in\mathbb{N}_0}$ is
\begin{equation}\label{CDF2}
1-(x-y)\sum_{k=0}^n
\pi_kQ_k(x)Q_k^{(0)}(y)
=
\lambda_n\pi_n
\big[
Q_{n+1}(x)Q_n^{(0)}(y)
-
Q_n(x)Q_{n+1}^{(0)}(y)
\big],
\quad n\geq0.
\end{equation}

\smallskip

Let us define the \emph{first passage time}, or \emph{hitting time}, of the
birth--death process to the state $j$ as
\[
\tau_j=\inf\{t\geq0:X_t=j\},
\quad
j\in\mathbb{N}_0\cup\{-1\}.
\]
The quantities $A$ and $B$ in
\eqref{quantity_AB} also characterize the long-term behavior of the
process:
\begin{itemize}[leftmargin=0.25in]

\item If $\mu_0=0$, the birth--death process is \emph{recurrent} if and
only if $A=\infty$; otherwise it is \emph{transient}. Moreover, in the
recurrent case, the process is \emph{positive recurrent} if and only if
$B<\infty$; otherwise it is \emph{null recurrent}.

\item If $\mu_0>0$, absorption at $-1$ is \emph{certain}, that is,
\(
\mathbb{P}_i(\tau_{-1}<\infty)=1,i\geq0,
\)
if and only if $A=\infty$; otherwise, absorption at $-1$ is not certain.
Furthermore, when absorption is certain, it is \emph{ergodic}, in the
sense that
\[
\mathbb{E}_i[\tau_{-1}]
:=
\mathbb{E}[\tau_{-1}\mid X_0=i]
<\infty,
\quad i\geq0,
\]
if and only if $B<\infty$; otherwise it is \emph{non-ergodic}.

\end{itemize}
These classifications can equivalently be characterized in terms of the
behavior of the spectral measure $d\psi$ near $x=0$; see
\cite[Theorems~3.42 and~3.51]{MDIB}.

\smallskip

Now, let \(h:\mathbb{N}_0 \to (0,\infty)\) be a sequence satisfying
\((\mathcal{A}h)(i) = 0\) for all \(i \in \mathbb{N}_0\). Then \(h\) is called a
\emph{harmonic} sequence with respect to \(\mathcal{A}\). From here, we can define the
\emph{Doob $h$-transform} of $\mathcal{A}$ (see \cite[Section 17.6]{LP}) as the operator
$\mathcal{A}_h$ defined by
\begin{equation}\label{doobt}
(\mathcal{A}_h f)(i)=\frac{1}{h(i)}\mathcal{A}(h f)(i),
\quad i \in \mathbb{N}_0,
\end{equation}
which is the infinitesimal generator of a new \emph{conservative} birth--death process with
modified transition rates
\[
\lambda_i^{h} = \lambda_i\,\frac{h(i+1)}{h(i)},\; i\geq0,
\quad 
\mu_i^{h} = \mu_i\,\frac{h(i-1)}{h(i)},\; i\geq1,\quad \mu_0^{h}=0.
\]
Indeed, the harmonicity of $h$ implies that $\lambda_i^h+\mu_i^h=\lambda_i+\mu_i$ for
$i\ge1$ and $\lambda_0^h=\lambda_0+\mu_0$. The process $\{X_t^{h}, t \ge 0\}$ governed by
$\mathcal{A}_h$ can often be interpreted as the original process $\{X_t,\, t \ge 0\}$
conditioned on an event determined by $h$; for instance, if $h(i)$ is proportional to the
probability of never being absorbed at $-1$ starting from $i$, then $\mathcal{A}_h$
describes the process conditioned on survival for all time. Doob \(h\)-transforms are
commonly used to describe Markov processes conditioned on boundary or survival events.

\begin{lemma}\label{lem21}
Let $\{X_t, t\geq0\}$ be the birth--death process with infinitesimal generator $\mathcal{A}$ given by \eqref{QQmm} with $\mu_0>0$. For some $n\in\mathbb{N}_0$, let $p(i)=\mathbb{P}_i(\tau_{n}<\tau_{-1}), i=-1,0,\ldots n$. Then the function $p$ is harmonic with respect to $\mathcal A$ on $\{0,\ldots,n-1\}$, i.e., it solves the Dirichlet problem $(\mathcal Ap)(i)=0$, $0\le i\le n-1$, with boundary values $p(-1)=0$, $p(n)=1$, and can be written as 
$$
p(i)=\mathbb{P}_i(\tau_{n}<\tau_{-1})=\frac{Q_i(0)}{Q_{n}(0)},\quad i=-1,0,\ldots n,
$$
where $Q_n(0)$ is given by \eqref{Qn0}.
\end{lemma}

\begin{proof}
For $i = -1, 0, \ldots, n$, consider the function $p(i) = \mathbb{P}_i(\tau_{n} < \tau_{-1})$, which satisfies the discrete harmonic relation
\begin{equation*}
\begin{split}
p(i) &= \frac{\lambda_i}{\lambda_i + \mu_i} \, p(i+1)
       + \frac{\mu_i}{\lambda_i + \mu_i} \, p(i-1),\quad i=0,1,\ldots,n-1,\\
p(-1)& = 0, \quad p(n) = 1.
\end{split}
\end{equation*}
Hence, \((\mathcal Ap)(i)=0\) for \(i=0,\ldots,n-1\), so \(p\) solves the stated Dirichlet problem. Solving this system with the given boundary conditions yields
\[
p(i) = 
\frac{1 + \mu_0 \displaystyle\sum_{k=0}^{i-1} \frac{1}{\lambda_k \pi_k}}
     {1 + \mu_0 \displaystyle\sum_{k=0}^{n-1} \frac{1}{\lambda_k \pi_k}},
\quad i = 0, 1, \ldots, n-1.
\]
Using \eqref{Qn0} the result follows. Alternatively, evaluating $x = 0$ in~\eqref{QQs}, one immediately observes that \(Q_i(0)/Q_n(0)\) satisfies the harmonic relation \((\mathcal Ap)(i)=0\) for \(i=0,\ldots,n-1\).
\end{proof}

Motivated by Lemma~\ref{lem21}, let us consider the positive function $h(i)=Q_i(0)$,
$i\in\mathbb{N}_0$. By evaluating \eqref{QQs} at $x=0$, we have $(\mathcal{A}h)(i)=0$ for
all $i\in\mathbb{N}_0$, so $h$ is harmonic with respect to $\mathcal{A}$. Note that if
$\mu_0=0$, then $h\equiv1$ and the transformation below is trivial, so it is only of
interest when $\mu_0>0$. The corresponding Doob $h$-transform yields a new birth--death
process $\{X_t^{h},t\geq0\}$ with infinitesimal generator $\mathcal{A}_h$ and modified rates
\begin{equation}\label{lamuh}
\lambda_i^{h}=\lambda_i\frac{Q_{i+1}(0)}{Q_i(0)},\quad i\geq0,
\quad
\mu_i^{h}=\mu_i\frac{Q_{i-1}(0)}{Q_i(0)},\quad i\geq1,
\quad
\mu_0^{h}=0.
\end{equation}
This transformation is also known as the \emph{Derman--Vere-Jones transformation} in the
context of birth--death processes (see \cite{MDIB}). When the probability of never being
absorbed is positive (that is, when $A<\infty$), Lemma~\ref{lem21} shows that
$h$ is proportional to the corresponding survival probability, since
$\mathbb{P}_i(\tau_n<\tau_{-1})=Q_i(0)/Q_n(0)\to Q_i(0)/(1+\mu_0A)$ as $n\to\infty$, and
the transformed process may therefore be interpreted as the original process
\emph{conditioned on never being absorbed}. When absorption occurs with probability one,
the same transform is understood instead as the limiting process obtained by conditioning
on reaching increasingly high levels before absorption; indeed, for every $n$, the original
process conditioned on $\{\tau_n<\tau_{-1}\}$ coincides, up to time $\tau_n$, with the
process generated by $\mathcal{A}_h$. This Doob $h$-transform has the following properties
(see \cite[Proposition 3.18]{MDIB}):
\begin{enumerate}[leftmargin=0.3in]
\item The spectral measure for $\mathcal{A}_h$ is the same as the spectral measure for $\mathcal{A}$, that is, $d\psi$.
\item If $P_{ij}^{h}(t)$ is the transition function for $\mathcal{A}_h$, then we have
$$
P_{ij}(t)Q_j(0)=Q_i(0)P_{ij}^{h}(t),\quad i,j\in\mathbb{N}_0.
$$
\item The potential coefficients $\pi_{j}^{h}$ for $\mathcal{A}_h$ are given by
\begin{equation}\label{dhtpc}
\pi_{j}^{h}=\pi_jQ_j^2(0),\quad j\in\mathbb{N}_0.
\end{equation}
\end{enumerate}

\begin{proposition}\label{prop22}
Let $\{X_t, t\geq0\}$ be the birth--death process with infinitesimal generator $\mathcal{A}$ given by \eqref{QQmm} with $\mu_0>0$. Then we have
\begin{equation}\label{endcod}
\mathbb{E}_n[\tau_{n+1} | \tau_{n+1}<\tau_{-1}]=\frac{1}{\lambda_n\pi_nQ_n(0)Q_{n+1}(0)}\sum_{k=0}^n\pi_kQ_k^2(0),\quad n\geq0.
\end{equation}
\end{proposition}
\begin{proof}
It is a direct consequence of the formula 
\begin{equation}\label{Expx}
\mathbb{E}_n[\tau_{n+1}] = \frac{1}{\lambda_n \pi_n} \sum_{k=0}^n \pi_k, 
\end{equation}
which can be found in \cite[Section 8.1]{And}, applied to the context of the new birth--death process generated by $\mathcal{A}_h$. Indeed, up to time $\tau_{n+1}$, the original process conditioned on
$\{\tau_{n+1}<\tau_{-1}\}$ is the Doob $h$-transform with
$h(i)=\mathbb{P}_i(\tau_{n+1}<\tau_{-1})=Q_i(0)/Q_{n+1}(0)$, which is proportional to $Q_i(0)$; hence it has birth--death rates \eqref{lamuh} and potential coefficients \eqref{dhtpc}.
\end{proof}

\begin{lemma}\label{lem23}
Let $\{X_t,t\geq0\}$ be a birth--death process on $\mathbb{N}_0$ with infinitesimal generator $\mathcal{A}$ given by \eqref{QQmm}. Let $
q(i)=\mathbb{P}_i(\tau_0<\infty), i\in\mathbb{N}_0.
$ Then $q$ is the minimal nonnegative solution of the Dirichlet problem
$$
\mathcal{A}q(i)=0,\quad i\geq1,
\quad q(0)=1.
$$
Equivalently, $q$ is harmonic on $\mathbb{N}_0$ with respect to the generator $\mathcal{A}_0$ of the process stopped upon reaching $0$,
defined by
$$
\mathcal{A}_0f(i)=\mathcal{A}f(i),\quad i\geq1,
\quad
\mathcal{A}_0f(0)=0.
$$
If $A<\infty$, where $A$ is defined by \eqref{quantity_AB}, then
\begin{equation}\label{qqprob}
q(i)
=
\mathbb{P}_i(\tau_0<\infty)
=
\frac{\displaystyle\sum_{k=i}^{\infty}
\dfrac{1}{\lambda_k\pi_k}}
{\displaystyle\sum_{k=0}^{\infty}
\dfrac{1}{\lambda_k\pi_k}}
=
1-\frac{1}{A}
\sum_{k=0}^{i-1}\frac{1}{\lambda_k\pi_k},
\quad i\in\mathbb{N}_0.
\end{equation}
If $A=\infty$, then
$
q(i)=1
$ for all $i\in\mathbb{N}_0$.
\end{lemma}
\begin{proof}
The proof follows a similar reasoning to that of Lemma~\ref{lem21}, but characterizing the sequence $q(i)=\mathbb{P}_i(\tau_0<\infty)$ as the \emph{minimal} nonnegative solution of the corresponding Dirichlet problem with boundary condition $q(0)=1$. Then the minimality condition gives \eqref{qqprob} when $A<\infty$. If $A=\infty$, the minimal nonnegative solution is the constant function $q\equiv1$. For further details, see \cite[Example~1.3.4]{Norr}.
\end{proof}

\begin{remark}
For an explosive process, the event $\{\tau_0=\infty\}$ in Lemma~\ref{lem23} also
contains the trajectories that escape to infinity in finite time without visiting $0$.
Since we work with the minimal process, these trajectories never come back, and
\eqref{qqprob} remains valid. For a non-minimal process returning from infinity, the
probability of eventually reaching $0$ would in general be larger. The same comment
applies to the event $\mathcal E_\omega$ in \eqref{eventEomega}, which is used in the
probabilistic interpretation of the UL factorization.
\end{remark}

\smallskip

Consider now the same birth--death process $\{X_t, t \ge 0\}$, but defined on the truncated state space $\mathcal{S}_n = \{-1, 0, \ldots, n\}$. The corresponding infinitesimal generator is obtained by taking the first $n+1$ rows and columns of $\mathcal{A}$ in~\eqref{QQmm} and setting $\lambda_n = 0$. For any $n \in \mathbb{N}_0$, define
\begin{equation*}
T_n^{(j)} := \int_{0}^{\tau_j} \mathbf{1}_{\{ X_s = n \}}\, ds, \quad j \in \mathcal{S}_n,
\end{equation*}
to be the \emph{occupation time} of state~$n$ before hitting state~$j$, where $\mathbf{1}_A$ is the indicator function; that is, $T_n^{(j)}$ represents the total time the process spends in state~$n$ prior to its first visit to state~$j$.

\begin{proposition}\label{prop23}
Let $\{X_t, t \ge 0\}$ be a birth--death process on $\mathcal{S}_n = \{-1, 0, \ldots, n\}$ with $\mu_0 > 0$, defined by the truncated infinitesimal generator $\mathcal{A}$ in~\eqref{QQmm}, where $\lambda_n = 0$. Then the expected occupation time of state~$n$ before absorption at~$-1$, starting from any state $i \in \{0, \ldots, n\}$, is given by
\begin{equation}\label{expocctime}
\mathbb{E}_i\left[T_n^{(-1)}\right] = \frac{\pi_n Q_i(0)}{\mu_0}=\pi_n\sum_{k=0}^i\frac{1}{\mu_k\pi_k}, \quad i \in \{0, \ldots, n\}.
\end{equation}
\end{proposition}
\begin{proof}
Let $u(i)=\mathbb{E}_i\big[T_n^{(-1)}\big]$. Conditioning on the first jump (first-step
analysis; see \cite[Section~3.3]{Norr}), and noting that the process spends an exponential holding time of mean $1/(\lambda_i+\mu_i)$ in state $i$, which contributes to $T_n^{(-1)}$ only if $i=n$, we obtain that $u$ satisfies the following system of equations:
\begin{equation*}
\begin{split}
(\mathcal{A}u)(i) &= -\mathbf{1}_{\{i= n\}}, \quad i=0,1,\ldots,n,\\ u(-1)& = 0.
\end{split}
\end{equation*}
Expanding these relations and recalling that $\lambda_n = 0$, we have
\begin{align*}
\lambda_i\big(u(i+1) - u(i)\big) + \mu_i\big(u(i-1) - u(i)\big) &= 0, \quad i = 0, 1, \ldots, n-1,\\
\mu_n\big(u(n-1) - u(n)\big) &= -1, \quad u(-1) = 0.
\end{align*}
Set $s_i = u(i) - u(i-1)$, so that $u(i) = \sum_{k=0}^i s_k$. From the last equation we obtain $s_n = 1 / \mu_n$. Substituting this into the recurrence and telescoping yields (see~\eqref{potcoef})
\[
s_{i+1} = \frac{\mu_i}{\lambda_i}s_i 
= \cdots 
= \frac{\mu_i \cdots \mu_0}{\lambda_i \cdots \lambda_0}s_0
= \frac{\mu_0 s_0}{\lambda_i \pi_i}.
\]
Using $s_n = 1 / \mu_n$ and the symmetry relation~\eqref{sympotcoef}, we obtain $s_0 = \pi_n / \mu_0$. Hence,
\[
s_i = \frac{\pi_n}{\mu_i \pi_i}
\quad \Longrightarrow \quad
u(i) = \pi_n \sum_{k=0}^i \frac{1}{\mu_k \pi_k},
\quad i = 0, 1, \ldots, n.
\]
Finally, using the definition of $Q_i(0)$ in~\eqref{Qn0} and again \eqref{sympotcoef}, we readily recover~\eqref{expocctime}.
\end{proof}
Here and below, $\mathbb{E}_m\big[T_m^{(-1)}\big]$ always refers to the process truncated
at level $m$, i.e.\ on $\mathcal{S}_m$ with $\lambda_m=0$, as in Proposition~\ref{prop23};
in particular, $\mathbb{E}_n[T_n^{(-1)}]$ and $\mathbb{E}_{n+1}[T_{n+1}^{(-1)}]$ refer
to two different truncations.

\begin{remark}
Proceeding as in Proposition~\ref{prop23}, but now considering the state space $\{0, 1, \ldots, n\}$ with $\mu_0 = 0$, we obtain
\begin{equation*}
\mathbb{E}_i\left[T_n^{(0)}\right]
= \pi_n \sum_{k=0}^{i-1} \frac{1}{\lambda_k \pi_k} 
= \pi_n \sum_{k=1}^{i} \frac{1}{\mu_k \pi_k},
\quad i = 0, 1, \ldots, n.
\end{equation*}
Using the property of the $0$-th associated polynomials (see~\cite[Section~3.3.2]{MDIB}),
\begin{equation}\label{Qn00}
Q_i^{(0)}(0) = -\sum_{k=0}^{i-1} \frac{1}{\lambda_k \pi_k}, 
\quad i \ge 1,
\end{equation}
it follows that $\mathbb{E}_i\left[T_n^{(0)}\right]$ can equivalently be expressed as
\begin{equation}\label{expocctime00}
\mathbb{E}_i\left[T_n^{(0)}\right] = -\pi_n Q_i^{(0)}(0), 
\quad i = 0, 1, \ldots, n.
\end{equation}
\end{remark}

\section{LU factorization and the Darboux transformation}\label{Sec3}

In this and the next section, we consider LU and UL factorizations of the infinitesimal generator $\mathcal{A}$ in~\eqref{QQmm}, namely
\[
\mathcal{A}=\widetilde{\mathcal{A}}_L\widetilde{\mathcal{A}}_U, 
\quad 
\mathcal{A}=\mathcal{A}_U\mathcal{A}_L,
\]
where $\mathcal{A}_L$ and $\widetilde{\mathcal{A}}_L$ are lower bidiagonal matrices, and $\mathcal{A}_U$ and $\widetilde{\mathcal{A}}_U$ are upper bidiagonal matrices. From the structure of the infinitesimal generator $\mathcal{A}$ in~\eqref{QQmm} we have that $\mathcal{A}\bm e=-\mu_0e^{(0)}$, where $\bm e$ denotes the semi-infinite vector of all ones and $e^{(0)}=(1,0,0,\ldots)^T$ is the first canonical semi-infinite vector. In principle, there are many possible factorizations of this type, as they always depend on free parameters. Our focus, however, is on those factorizations whose associated Darboux transformations, namely
\[
\widehat{\mathcal{A}}=\widetilde{\mathcal{A}}_U\widetilde{\mathcal{A}}_L, 
\quad 
\widetilde{\mathcal{A}}=\mathcal{A}_L\mathcal{A}_U,
\]
are \emph{again} infinitesimal generators of some birth--death process, that is, they have positive off-diagonal birth--death rates and they satisfy $\widehat{\mathcal{A}}\bm e=-\hat{\mu}_0e^{(0)}$ and $\widetilde{\mathcal{A}}\bm e=-\tilde{\mu}_0e^{(0)}$, for some $\hat{\mu}_0,\tilde{\mu}_0\geq0$. This approach differs from the discrete-time case \cite{GdI}, where both the original and Darboux-transformed transition matrices were stochastic matrices with no absorption. In addition, we will be interested in providing a probabilistic interpretation of the entries of the corresponding factors. In this section we study the LU factorization, which in this setting is uniquely determined once the absorption rate of the Darboux-transformed process is fixed, while the UL factorization, where one free parameter appears, is studied in Section~\ref{Sec4}.

\medskip

Consider $\mathcal{A}=\widetilde{\mathcal{A}}_L\widetilde{\mathcal{A}}_U$, where

\begin{equation}\label{AL_AU}
\widetilde{\mathcal{A}}_L=
		\begin{pmatrix}
			\tilde{s}_0 &  &  &  \\
			\tilde{r}_1 & \tilde{s}_1 &  &  \\
			 & \tilde{r}_2 &\tilde{s}_2 &  \\
			 & & \ddots    & \ddots
		\end{pmatrix},\quad
		\widetilde{\mathcal{A}}_U=
		\begin{pmatrix}
			\tilde{y}_0 & \tilde{x}_0 &  &  & \\
			 & \tilde{y}_1 & \tilde{x}_1 &  & \\
			&  &\tilde{y}_2 &  \tilde{x}_2 & \\
			 &  &  & \ddots & \ddots
		\end{pmatrix}.
\end{equation}
This gives the system of equations
\begin{equation}\label{AL_AU_eq}
\begin{split}
\lambda_n & = \tilde{s}_n\tilde{x}_n,\quad n\geq0, \\	
-(\mu_n+\lambda_n) & = \tilde{r}_n\tilde{x}_{n-1}+\tilde{s}_n\tilde{y}_n, \quad n\geq1, \quad -(\lambda_0+\mu_0)=\tilde{s}_0\tilde{y}_0,\\
\mu_n &= \tilde{r}_n\tilde{y}_{n-1}, \quad n\geq1.
\end{split}
\end{equation}
Since the birth and death rates $\{\lambda_n,\mu_n\}$ are positive, we have that $\tilde{s}_n\tilde{x}_n>0, n\geq 0,$ and $\tilde{r}_n\tilde{y}_{n-1}>0, n\geq 1$; hence, $\tilde{s}_n,\tilde{x}_n\neq 0, n\geq 0,$ and $\tilde{r}_n,\tilde{y}_{n-1}\neq 0, n\geq 1$. Additionally, we are interested in the case where the Darboux transformation $\widehat{\mathcal{A}}=\widetilde{\mathcal{A}}_U\widetilde{\mathcal{A}}_L $ produces a new birth--death process. That means that we need to have that $\widehat{\mathcal{A}}\bm e=-\hat{\mu}_0e^{(0)}$, that is,
\begin{equation}\label{AU_AL_eq}
\begin{split}
\hat{\lambda}_n & = \tilde{x}_n\tilde{s}_{n+1}, \quad n\geq0, \\
-(\hat{\lambda}_n+\hat{\mu}_n) & = \tilde{x}_n\tilde{r}_{n+1}+\tilde y_n\tilde{s}_{n},\quad n\geq 0,\\
\hat{\mu}_n &= \tilde{y}_n\tilde{r}_{n}, \quad n\geq1,
\end{split}
\end{equation}
with $\hat\lambda_n,\hat\mu_{n+1}>0, n\geq0,$ and $\hat\mu_0\geq0$. 
\smallskip

Before continuing, we may assume without loss of generality that $\tilde s_0=1$. This is a consequence of the normalization of the polynomials generated by the three-term recurrence relation associated with $\widehat{\mathcal{A}}$. Indeed, take the vector of polynomials $Q$ generated by the eigenvalue equation $-xQ = \mathcal{A}Q$. Using the factorization $\mathcal{A}=\widetilde{\mathcal{A}}_L\widetilde{\mathcal{A}}_U$ and multiplying $\widetilde{\mathcal{A}}_U$ on the left we obtain
\(-x \widetilde{\mathcal{A}}_U Q = \widetilde{\mathcal{A}}_U \widetilde{\mathcal{A}}_L\widetilde{\mathcal{A}}_U Q.\) Defining $\bar{Q} = \widetilde{\mathcal{A}}_U Q$, it follows that
\(-x \bar{Q} = \widehat{\mathcal{A}}\,\bar{Q},\) and the family $(\bar{Q}_n)_{n\in\mathbb{N}_0}$ can be obtained from $(Q_n)_{n\in\mathbb{N}_0}$ as follows:
\begin{equation}\label{Qb1}
\bar{Q}_n(x) = \tilde{x}_nQ_{n+1}(x)+\tilde{y}_nQ_n(x), 
\quad n \geq 0.
\end{equation}
Moreover, if $-xQ = \widetilde{\mathcal{A}}_L\widetilde{\mathcal{A}}_U Q,$ we also have
\(-xQ = \widetilde{\mathcal{A}}_L \bar{Q},\) that is,
\begin{equation}\label{Qb2}
\begin{aligned}
-x Q_0(x) &= \tilde{s}_0 \, \bar{Q}_0(x), \\
-x Q_n(x) &= \tilde{r}_n \, \bar{Q}_{n-1}(x) + \tilde{s}_n \, \bar{Q}_n(x), 
\quad n \geq 1.
\end{aligned}
\end{equation}
In this case, it is important to note that $\deg(\bar{Q}_n)=n+1$ and, using \eqref{Qb2}, we have that $(\bar{Q}_n)_{n\in\mathbb{N}_0}$ can be written as
\begin{equation}\label{Qb3}
\bar{Q}_0(x) = -\frac{x}{\tilde s_0}=-x\hat{Q}_0(x), 
\quad 
\bar{Q}_n(x) = -x \hat{Q}_n(x),\quad n\geq1,
\end{equation}
where $(\hat{Q}_n)_{n\in\mathbb{N}_0}$ is a family of polynomials satisfying $\deg(\hat{Q}_n) = n$. Since we wish to normalize this family so that $\hat{Q}_0(x) = 1$, we may assume that $\tilde{s}_0 = 1$. This will be the family of polynomials associated with the Darboux transformation $\widehat{\mathcal{A}}$.

As for the spectral measure associated with $\widehat{\mathcal{A}}$, it is well known (see \cite{GdI,Yo}) that the sequence of polynomials $(\hat{Q}_n)_{n\in\mathbb{N}_0}$ defined by \eqref{Qb3} is orthogonal with respect to the \emph{Christoffel transform} of $d\psi$ given by 
\begin{equation}\label{spmec1}
d\widehat{\psi}(x)=\frac{x}{\lambda_0+\mu_0}\,d\psi(x).
\end{equation}
By integrating \eqref{QQs} for $n=0$ with respect to $d\psi$, we see that the normalizing constant $\lambda_0+\mu_0$ ensures that $d\widehat\psi$ is a probability measure. Also, the potential coefficients $(\hat{\pi}_n)_{n\in\mathbb{N}_0}$ of $\widehat{\mathcal{A}}$ are given by
\begin{equation*}
\hat{\pi}_0=1, \quad\hat{\pi}_n=\frac{\hat{\lambda}_0\cdots\hat{\lambda}_{n-1}}{\hat{\mu}_1\cdots\hat{\mu}_n}=\frac{\tilde{y}_0\tilde{s}_n}{\tilde{y}_n}\pi_n,  \quad n\geq 1.
\end{equation*}
The construction of the spectral measure and the orthogonal polynomials associated with the Darboux transformation is similar to that carried out in the discrete-time case \cite{GdI} (see also \cite{Yo}).

\medskip

Going back to the LU factorization, let us now give a procedure to compute the coefficients $\tilde s_n,\tilde r_n,\tilde y_n, \tilde x_n,$ in \eqref{AL_AU} recursively, in this order, using \eqref{AL_AU_eq} and \eqref{AU_AL_eq}. For that, introduce the sequences
\begin{equation}\label{ttqq}
 \tilde t_0=1,\quad \tilde t_n= \tilde s_n+\tilde r_n,\quad n\geq1,\quad \tilde q_n = \frac{\tilde y_n}{\tilde x_n},\quad n\geq0.
\end{equation}
Summing the three equations in \eqref{AU_AL_eq} and taking into account that $\tilde s_0=1$, we obtain the coupling conditions:
\begin{equation*}
\begin{split}
&\tilde y_0+\tilde x_0(\tilde s_1+\tilde r_1)=-\hat\mu_0,\\
&\tilde x_n(\tilde s_{n+1}+\tilde r_{n+1})+\tilde y_n(\tilde s_n+\tilde r_n)=0,\quad n\ge1.
\end{split}
\end{equation*}
With the definitions \eqref{ttqq}, and using that $\tilde x_0=\lambda_0$ and $\tilde y_0=-(\lambda_0+\mu_0)$ (see \eqref{AL_AU_eq} for $n=0$), we may write these coupling conditions as the first-order difference equation
\begin{equation}\label{deqtt}
\tilde t_{n+1} = -\tilde q_n\tilde t_n,\quad n\geq1,\quad \tilde t_1=1+\frac{\mu_0-\hat\mu_0}{\lambda_0}.
\end{equation}
On the other hand, substituting \(\tilde r_n=\dfrac{\mu_n}{\tilde y_{n-1}}\) and \(\tilde s_n=\dfrac{\lambda_n}{\tilde x_n}\) into the second relation in \eqref{AL_AU_eq}, gives
\[
-(\lambda_n+\mu_n)
= \frac{\mu_n}{\tilde q_{n-1}} + \lambda_n \tilde q_n.
\]
Rearranging this identity we get the nonlinear first-order recurrence equation
\begin{equation}\label{deqqq}
\tilde q_n= -1 - \frac{\mu_n}{\lambda_n} - \frac{\mu_n}{\lambda_n\tilde q_{n-1}},\quad n\ge1,\quad \tilde q_0 = -1-\frac{\mu_0}{\lambda_0}.
\end{equation}

\begin{proposition}
The explicit solutions of the recurrence equations \eqref{deqtt} and \eqref{deqqq} are given by
\begin{equation}\label{solqt}
\begin{split}
\tilde q_n&= -\frac{Q_{n+1}(0)}{Q_n(0)},\quad n\geq0, \\
\tilde t_n &= \left(1-\frac{\hat\mu_0}{\lambda_0+\mu_0}\right)Q_n(0),\quad n\geq2,\quad \tilde t_1=1+\frac{\mu_0-\hat\mu_0}{\lambda_0},\quad \tilde t_0=1,
\end{split}
\end{equation}
where $Q_n(0)$ is defined by \eqref{Qn0}.
\end{proposition}
\begin{proof}
Evaluating \eqref{QQs} at $x=0$ and dividing by $-Q_n(0)$ gives
\begin{equation*}
-\frac{Q_{n+1}(0)}{Q_n(0)} = -1-\frac{\mu_n}{\lambda_n}+\frac{\mu_n}{\lambda_n}\frac{Q_{n-1}(0)}{Q_n(0)}.
\end{equation*}
We observe that the sequence $-Q_{n+1}(0)/Q_{n}(0)$ satisfies the nonlinear first-order recurrence equation with initial condition \eqref{deqqq}. Hence we get the solution for $\tilde q_n$ by induction. On the other hand, solving directly \eqref{deqtt}, we obtain
\[
\tilde t_n=\tilde t_1\prod_{k=1}^{n-1}(-\tilde q_k), \quad n\ge2.
\]
Using the representation for $\tilde q_n$, the initial value $\tilde t_1$ in \eqref{deqtt} and that $Q_1(0)=1+\mu_0/\lambda_0$, the product telescopes and gives the closed form
$$
\tilde t_n =\left(1+\frac{\mu_0-\hat\mu_0}{\lambda_0}\right)\frac{Q_n(0)}{Q_1(0)}=\left(1-\frac{\hat\mu_0}{\lambda_0+\mu_0}\right)Q_n(0),\quad n\geq2,
$$
obtaining the result.
\end{proof}
If the sequences $\tilde t_n$ and $\tilde q_n$ are known, then
\begin{equation*}
\tilde r_n =\frac{\mu_n}{\tilde y_{n-1}}=\frac{\mu_n}{\tilde q_{n-1} \tilde x_{n-1}}= \frac{\mu_n \tilde s_{n-1}}{\lambda_{n-1}\tilde q_{n-1}}\quad \Longrightarrow\quad
\tilde s_n= \tilde t_n - \frac{\mu_n \tilde s_{n-1}}{\lambda_{n-1}\tilde q_{n-1}},\quad n\geq1.
\end{equation*}
This gives a recursive way of obtaining explicitly $\tilde s_n$ knowing that $\tilde s_0=1$. Indeed,
$$
\tilde s_n =\left(1-\frac{\hat\mu_0}{\lambda_0+\mu_0}\right)Q_n(0) + \frac{\mu_n }{\lambda_{n-1}}\frac{Q_{n-1}(0)}{Q_{n}(0)}\tilde s_{n-1},\quad n\geq1,\quad\tilde s_0=1.
$$
A straightforward computation using the definition of the potential coefficients $(\pi_n)_{n\in\mathbb{N}_0}$ in \eqref{potcoef} gives
\begin{equation}\label{ssnn}
\tilde s_n=\frac{1}{\pi_n Q_n(0)}\left[\frac{\hat\mu_0}{\lambda_0+\mu_0}+\left(1-\frac{\hat\mu_0}{\lambda_0+\mu_0}\right)\sum_{k=0}^n \pi_kQ_k^2(0)\right], \quad n\ge1,\quad \tilde s_0=1.
\end{equation}
Once we have $\tilde s_n$ and using the initial conditions $\tilde s_0=1$ and $\tilde x_0=\lambda_0,\tilde y_0=-(\lambda_0+\mu_0)$, the rest of sequences $\tilde r_n,\tilde x_n,\tilde y_n$ are obtained by the following algebraic relations:
\begin{equation}\label{algrelLU}
\tilde r_n= \tilde t_n - \tilde s_n,\quad n\geq1,\quad \tilde x_n = \frac{\lambda_n}{\tilde s_n}, \quad
\tilde y_n = \tilde q_n \tilde x_n=\frac{\lambda_n\tilde q_n}{\tilde s_n},\quad n\geq0.
\end{equation}
In other words,
\begin{align}
\label{rrnn}\tilde r_n&=-\frac{1}{\pi_n Q_n(0)}\left[\frac{\hat\mu_0}{\lambda_0+\mu_0}+\left(1-\frac{\hat\mu_0}{\lambda_0+\mu_0}\right)\sum_{k=0}^{n-1} \pi_kQ_k^2(0)\right], \quad n\ge1,\\
\label{xxnn}\tilde x_n&=\frac{\lambda_n\pi_nQ_n(0)}{\dfrac{\hat\mu_0}{\lambda_0+\mu_0}+\left(1-\dfrac{\hat\mu_0}{\lambda_0+\mu_0}\right)\D\sum_{k=0}^n \pi_kQ_k^2(0)},\quad n\ge1,\quad \tilde x_0=\lambda_0,\\
\label{yynn}\tilde y_n&=-\tilde{x}_n\frac{Q_{n+1}(0)}{Q_n(0)}=-\frac{\lambda_n\pi_nQ_{n+1}(0)}{\dfrac{\hat\mu_0}{\lambda_0+\mu_0}+\left(1-\dfrac{\hat\mu_0}{\lambda_0+\mu_0}\right)\D\sum_{k=0}^n \pi_kQ_k^2(0)},\quad n\ge1,\quad \tilde y_0=-(\lambda_0+\mu_0).
\end{align}
From \eqref{AU_AL_eq} and \eqref{algrelLU} we can write the birth--death rates of the Darboux transformation $\widehat{\mathcal{A}}$ in terms of the original birth--death rates and the sequences $\tilde s_n$ and $\tilde q_n$, that is
\begin{equation}\label{newbdcoe}
\hat\lambda_n=\lambda_n\frac{\tilde s_{n+1}}{\tilde s_n},\quad\hat\mu_{n+1}=\frac{\lambda_{n+1}\mu_{n+1}}{\lambda_n}\cdot\frac{\tilde q_{n+1}}{\tilde q_n}\cdot\frac{\tilde s_n}{\tilde s_{n+1}},\quad n\geq0.
\end{equation}

\begin{remark}
In \cite{GrH1,GrH3,GrHH}, the LU and UL factorizations were studied for monic Jacobi matrices, whereas here they are formulated directly for infinitesimal operators of the form $\mathcal{A}$ in \eqref{QQmm}. The two settings are equivalent up to the standard normalization relating a birth--death generator to its associated monic Jacobi matrix, and therefore the corresponding factorizations can be translated from one formulation to the other. Nevertheless, we work directly with $\mathcal{A}$ and analyze both factorizations explicitly, since this yields closed-form expressions for the coefficients of the factors in terms of the birth and death rates. These expressions are needed for the subsequent probabilistic interpretation studied in Sections \ref{piclu} and \ref{picul} below and for the analysis of the associated Darboux transformations.
\end{remark}

\begin{remark}
Observe that, if we use the formula \eqref{Qn0} for $Q_n(0)$ we can write explicitly all coefficients $\tilde r_n,\tilde s_n, \tilde x_n,\tilde y_n$ in terms of the birth--death rates $\{\lambda_n,\mu_n\}$ and the corresponding potential coefficients $\pi_n$ in \eqref{potcoef}. This improves the discrete-time case \cite{GdI}, where the factor coefficients were only obtained recursively, rather than explicitly as done here. 
\end{remark}

\begin{remark}\label{remdual}
Observe that if we choose $\mu_0=0$ and $\hat\mu_0=\lambda_0$ we obtain a simplified LU factorization with coefficients (taking into account that $Q_n(0)=1$ in this case, see \eqref{Qn0}):
$$
\tilde s_n=\frac{1}{\pi_n},\quad\tilde r_{n+1}=-\frac{1}{\pi_{n+1}},\quad \tilde x_n=\lambda_n\pi_n,\quad \tilde y_n=-\lambda_n\pi_n,\quad n\geq0.
$$
The birth--death rates of the Darboux transformation $\widehat{\mathcal{A}}$ (or $\widehat{\mathcal{A}}^{(\lambda_0)}$
in the notation used at the beginning of Section~\ref{sec2}) are given by (see \eqref{newbdcoe}) 
$
\hat\lambda_n=\mu_{n+1},\hat\mu_{n}=\lambda_{n},n\geq0.
$
This factorization is closely related to the family of \emph{dual polynomials} $(H_n)_{n\in\mathbb{N}_0}$ already introduced in \cite{KMc2} (for the case $\mu_0=0$). It is easy to see that in our case, these dual polynomials coincide with the family $(\hat Q_n)_{n\in\mathbb{N}_0}$ defined by \eqref{Qb3}. For more information, see \cite[Remark 3.15]{MDIB}, where $\widehat{\mathcal{A}}^{(\lambda_0)}$ is denoted by $\mathcal{A}^d$.
\end{remark}

\begin{remark}
The family of orthogonal polynomials $(\hat{Q}_n)_{n\in\mathbb{N}_0}$ associated with the Darboux transformation $\widehat{\mathcal{A}}$, and orthogonal with respect to the Christoffel transform of $d\psi$ given in~\eqref{spmec1}, can be simplified by means of the Christoffel--Darboux formula~\eqref{CDF}. Indeed, recall that this family can be defined, using~\eqref{Qb1} and~\eqref{Qb3}, as
\begin{equation}\label{Qnhat}
\hat{Q}_n(x) = -\frac{1}{x}\left(\tilde{x}_n Q_{n+1}(x) + \tilde{y}_n Q_n(x)\right), \quad n \ge 0.
\end{equation}
Since $\tilde{y}_n = \tilde{q}_n \tilde{x}_n$, and using the definition of $\tilde{q}_n$ in~\eqref{solqt}, we obtain
\begin{align}
\hat{Q}_n(x)
\nonumber&= -\frac{\tilde{x}_n}{x} \left( Q_{n+1}(x) - \frac{Q_{n+1}(0)}{Q_n(0)} Q_n(x) \right)
= -\frac{\tilde{x}_n Q_{n+1}(0)}{x} \left( \frac{Q_{n+1}(x)}{Q_{n+1}(0)} - \frac{Q_n(x)}{Q_n(0)} \right) \\[0.3em]
\nonumber&= \frac{\tilde{x}_n}{Q_n(0)} \left[ \frac{Q_{n+1}(x) Q_n(0) - Q_n(x) Q_{n+1}(0)}{-x} \right]
= \frac{\tilde{x}_n}{\lambda_n \pi_n Q_n(0)} \sum_{k=0}^n \pi_k Q_k(0) Q_k(x) \\[0.3em]
\label{Qnhatsimp}&= \frac{\displaystyle\sum_{k=0}^n \pi_k Q_k(0) Q_k(x)}{\dfrac{\hat{\mu}_0}{\lambda_0 + \mu_0} + \left(1 - \dfrac{\hat{\mu}_0}{\lambda_0 + \mu_0}\right) \displaystyle\sum_{k=0}^n \pi_k Q_k^2(0)}.
\end{align}
The last two steps follow from the Christoffel--Darboux formula~\eqref{CDF} evaluated at $y = 0$, together with the definition of $\tilde{x}_n$ given in~\eqref{xxnn}.
\end{remark}

Up to now, we have computed the coefficients $(\tilde s_n)_{n\in\mathbb{N}_0}$, $(\tilde r_n)_{n\in\mathbb{N}}$, $(\tilde y_n)_{n\in\mathbb{N}_0}$, and $(\tilde x_n)_{n\in\mathbb{N}_0}$ in \eqref{ssnn} and \eqref{rrnn}--\eqref{yynn} so that $\widehat{\mathcal{A}}\bm e=-\hat{\mu}_0e^{(0)}$, that is, the sum of every row of the Darboux transform \(\widehat{\mathcal{A}}\) equals 0 (except for the first one, see \eqref{AU_AL_eq}). 
Now, we are ready to state the first of the main results of this paper, which provides conditions ensuring that the Darboux transform \(\widehat{\mathcal{A}}\) defines a genuine birth--death process.

\begin{theorem}\label{theom33}
Let $\{X_t : t \ge 0\}$ be a birth--death process with birth and death rates 
$\{\lambda_n,\mu_n\}$ and infinitesimal generator $\mathcal{A}$ given by~\eqref{QQmm}. Assume the standing condition \eqref{uniqueness_assumption}. Consider the LU factorization $\mathcal{A}=\widetilde{\mathcal{A}}_L\widetilde{\mathcal{A}}_U$, 
where $\widetilde{\mathcal{A}}_L$ and $\widetilde{\mathcal{A}}_U$ are defined in~\eqref{AL_AU}, 
with coefficients $(\tilde s_n)_{n\in\mathbb{N}_0}$, $(\tilde r_n)_{n\in\mathbb{N}}$, 
$(\tilde y_n)_{n\in\mathbb{N}_0}$, and $(\tilde x_n)_{n\in\mathbb{N}_0}$ given by \eqref{ssnn} and \eqref{rrnn}--\eqref{yynn}, respectively.  
Let \(\widehat{\mathcal{A}}=\widetilde{\mathcal{A}}_U \widetilde{\mathcal{A}}_L\) be the Darboux 
transform of $\mathcal{A}$ (see \eqref{AU_AL_eq}). Then $\widehat{\mathcal A}$ is the generator of a new birth--death process $\{\widehat{X}_t : t \ge 0\}$ with birth--death rates $\{\hat\lambda_n,\hat\mu_n\}$ given by~\eqref{newbdcoe} (with $(\tilde q_n)_{n\in\mathbb{N}_0}$ defined by \eqref{solqt}), if and only if the absorption rate $\hat\mu_0$ of the Darboux-transformed process satisfies
\begin{enumerate}
\item if $\mu_0>0$:
\begin{equation}\label{condLU0}
    0 \le \hat\mu_0 \le \lambda_0+\mu_0;
\end{equation}
\item if $\mu_0=0$:
\begin{equation}\label{condLU2}
0\le\hat\mu_0\le
\begin{cases}
\dfrac{\lambda_0B}{B-1}, & B<\infty,\\[2mm]
\lambda_0, & B=\infty,
\end{cases}
\end{equation}
where $B$ is defined by \eqref{quantity_AB}.
\end{enumerate}
\end{theorem}
\begin{proof}
By the computation of the coefficients $(\tilde s_n)_{n\in\mathbb{N}_0}$, $(\tilde r_n)_{n\in\mathbb{N}}$, $(\tilde y_n)_{n\in\mathbb{N}_0}$, and $(\tilde x_n)_{n\in\mathbb{N}_0}$ in \eqref{ssnn} and \eqref{rrnn}--\eqref{yynn} we have that $\widehat{\mathcal{A}}\bm e=-\hat{\mu}_0e^{(0)}$. \(\widehat{\mathcal{A}}\) defines a genuine birth--death process if and only if \(\hat{\lambda}_n, \hat{\mu}_{n+1} > 0\) for all \(n \ge 0\) and \(\hat{\mu}_0 \ge 0\). From \eqref{newbdcoe} and \eqref{solqt}, it follows that $\hat\lambda_n, \hat\mu_{n+1} > 0$ for all $n \ge 0$ if and only if $\tilde{s}_n > 0$ for all $n \ge 0$ (since $\tilde q_n<0$ and $\tilde s_0=1>0$). 
Moreover, the condition $\tilde{s}_n > 0$ for all $n \ge 0$ in \eqref{ssnn} imposes a restriction on the admissible values of $\hat\mu_0$, which must lie below a certain bound. 
Indeed, solving the inequalities $\tilde{s}_n > 0$ yields
$$
0\leq\hat\mu_0<(\lambda_0+\mu_0)\left[1+\left(\sum_{k=1
}^n\pi_kQ_k^2(0)\right)^{-1}\right],\quad n\geq1.
$$
Observe that the sequence \(\left(\sum_{k=1}^n\pi_kQ_k^2(0)\right)_{n\ge1}\) is always increasing. If \(\mu_0>0\), then \(\sum_{k=1}^\infty \pi_kQ_k^2(0)=S-1\), where \(S\) is defined in \eqref{SSS}. Since, by assumption, \(S=\infty\), condition \eqref{condLU0} follows immediately. On the other hand, if \(\mu_0=0\), then \(Q_n(0)=1\) for all \(n\geq0\), and therefore, letting \(n\to\infty\), we obtain \(\sum_{k=1}^\infty \pi_k = B-1\) (see \eqref{quantity_AB}). Consequently, condition \eqref{condLU2} follows. Note that \(B>1\), since \(\pi_0=1\) and $\pi_1>0$. The quantity $B$ in this context may be either finite or infinite.
\end{proof}

\begin{remark}
In the discrete-time setting, whenever a stochastic LU factorization is available, reversing the two stochastic factors automatically yields another stochastic matrix. In contrast, for continuous-time birth--death processes this property is far from automatic, since we allow possible absorption rates both in the original process and in its Darboux transformation. Consequently, additional conditions are required, and the previous theorem provides a genuinely new contribution specific to the continuous-time setting.
\end{remark}

Theorem~\ref{theom33} guarantees that $\widehat{\mathcal A}$ is the generator of a
birth--death process, but not that this new process again satisfies the standing
condition \eqref{uniqueness_assumption}. The following result characterizes when this is
the case, in terms of a single series depending only on the original birth--death rates.

\begin{proposition}\label{propSLU}
Under the assumptions of Theorem~\ref{theom33}, let $\hat\mu_0$ be admissible and set
\[
\Sigma_n=\sum_{k=0}^n\pi_kQ_k^2(0),\quad
\tilde d_n=\frac{\hat\mu_0}{\lambda_0+\mu_0}+\left(1-\frac{\hat\mu_0}{\lambda_0+\mu_0}\right)\Sigma_n,
\quad n\ge0.
\]
Then $\widehat{\mathcal A}$ satisfies the standing condition \eqref{uniqueness_assumption}
if and only if
\begin{equation}\label{ShatLU}
\widehat S=(\lambda_0+\mu_0)\sum_{n=0}^{\infty}\frac{\Sigma_n^2}{\lambda_n\pi_nQ_n(0)Q_{n+1}(0)}=\infty .
\end{equation}
In particular, this condition does not depend on $\hat\mu_0$. 
\end{proposition}

\begin{proof}
By \eqref{ssnn} and \eqref{yynn} we have, for all $n\ge0$ (including $n=0$, since $\tilde s_0=1$ and
$\tilde y_0=-(\lambda_0+\mu_0)$),
\[
\tilde s_n=\frac{\tilde d_n}{\pi_nQ_n(0)},\quad
\tilde y_n=-\frac{\lambda_n\pi_nQ_{n+1}(0)}{\tilde d_n}.
\]
Note that $\tilde d_n>0$ for all $n\ge0$, since $\tilde s_n>0$ by Theorem~\ref{theom33}.
Substituting into $\hat\pi_n=\tilde y_0\tilde s_n\pi_n/\tilde y_n$ (see the formula after
\eqref{spmec1}) gives
\begin{equation}\label{pihatLU}
\hat\pi_n=\frac{(\lambda_0+\mu_0)\,\tilde d_n^{\,2}}{\lambda_n\pi_nQ_n(0)Q_{n+1}(0)},\quad n\ge0 .
\end{equation}
Moreover, using $\hat\lambda_n=\lambda_n\tilde s_{n+1}/\tilde s_n$ from \eqref{newbdcoe}, we get
\begin{equation}\label{lpihatLU}
\hat\lambda_n\hat\pi_n=\frac{(\lambda_0+\mu_0)\,\tilde d_n\tilde d_{n+1}}{\pi_{n+1}Q_{n+1}^2(0)},\quad n\ge0 .
\end{equation}
The last expression in \eqref{Qnhatsimp} is an identity between polynomials, so it may be evaluated at $x=0$. Since its denominator is
$\tilde d_n$, we obtain
\begin{equation}\label{QhatLU0}
\hat Q_n(0)=\frac{\Sigma_n}{\tilde d_n},\quad n\ge0 .
\end{equation}
Now, from \eqref{pihatLU} and \eqref{QhatLU0},
\[
\widehat S=\sum_{n=0}^\infty\hat\pi_n\hat Q_n^2(0)
=(\lambda_0+\mu_0)\sum_{n=0}^{\infty}\frac{\Sigma_n^2}{\lambda_n\pi_nQ_n(0)Q_{n+1}(0)},
\]
which does not depend on $\hat\mu_0$. If $\hat\mu_0>0$,
condition \eqref{uniqueness_assumption} for $\widehat{\mathcal A}$ is precisely
$\widehat S=\infty$. If $\hat\mu_0=0$, then $\tilde d_n=\Sigma_n$, so that
$\hat Q_n(0)=1$ by \eqref{QhatLU0} (in agreement with \eqref{Qn0}), and hence
$\widehat B=\sum_n\hat\pi_n=\widehat S$. On the other hand, since
$\Sigma_{n+1}-\Sigma_n=\pi_{n+1}Q_{n+1}^2(0)$, \eqref{lpihatLU} gives
\[
\frac{1}{\hat\lambda_n\hat\pi_n}=\frac{1}{\lambda_0+\mu_0}\left(\frac{1}{\Sigma_n}-\frac{1}{\Sigma_{n+1}}\right),
\]
and therefore, since $\Sigma_0=1$ and $\Sigma_n\uparrow S$,
\[
\widehat A=\sum_{n=0}^\infty\frac{1}{\hat\lambda_n\hat\pi_n}=\frac{1}{\lambda_0+\mu_0}\left(1-\frac1S\right)<\infty .
\]
Hence $\widehat A+\widehat B=\infty$ if and only if $\widehat S=\infty$. In both cases,
\eqref{uniqueness_assumption} for $\widehat{\mathcal A}$ is equivalent to \eqref{ShatLU}.
\end{proof}

\begin{remark}
Condition \eqref{ShatLU} is not automatic. For instance, for $\mu_0=0$,
$\lambda_n=3\cdot2^n$ ($n\ge0$) and $\mu_n=2^n$ ($n\ge1$), we have $\pi_n=(3/2)^n$,
$\Sigma_n=2\big((3/2)^{n+1}-1\big)$ and $\lambda_n\pi_n=3^{n+1}$, so that
$\widehat S<\infty$, although the original birth--death process satisfies
\eqref{uniqueness_assumption} since $B=\infty$. 
\end{remark}

\subsection{Probabilistic interpretation of the coefficients of the LU factorization}\label{piclu} 

Assume first that \(\mu_0>0\). Using the results obtained in Section \ref{sec2} we may write the coefficients $\tilde{s}_n$, $\tilde{r}_n$, $\tilde{x}_n$, and $\tilde{y}_n$ in terms of expected values associated with certain stopping times. Indeed, from the definition of $\tilde{s}_n$ in~\eqref{ssnn}, and taking into account~\eqref{endcod} and~\eqref{expocctime}, we obtain $\tilde s_0=1$ and
\begin{equation}\label{ssnnprob}
\tilde{s}_n
= \left(1 - \dfrac{\hat{\mu}_0}{\lambda_0 + \mu_0}\right)
  \lambda_n Q_{n+1}(0)\,
  \mathbb{E}_n[\tau_{n+1} \mid \tau_{n+1} < \tau_{-1}] 
  + \frac{\hat{\mu}_0}{
    \mu_0 (\lambda_0 + \mu_0)\, \mathbb{E}_n\left[T_n^{(-1)}\right]},\quad n\geq1,
\end{equation}
where $\tau_j$ denotes the first passage time to state $j$, and $T_n^{(j)}$ is the occupation time of state $n$ before hitting state $j$, assuming the state space is $\mathcal{S}_n = \{-1, 0, \ldots, n\}$. Now, from the value of $\tilde t_n$ in \eqref{solqt} and using formula \eqref{endcod}, after some computations we obtain
\begin{equation}\label{rrnnprob}
\tilde{r}_n
=-\left(1 - \dfrac{\hat{\mu}_0}{\lambda_0 + \mu_0}\right)
  \mu_n Q_{n-1}(0)\,
  \mathbb{E}_{n-1}[\tau_{n} \mid \tau_{n} < \tau_{-1}] 
  - \frac{\hat{\mu}_0}{
    \mu_0 (\lambda_0 + \mu_0)\, \mathbb{E}_n\left[T_n^{(-1)}\right]},\quad n\geq1.
\end{equation}
Since $\tilde x_n=\lambda_n/\tilde{s}_n$, we get
\begin{equation}\label{xxnnprob}
\tilde{x}_n
= \left[\left(1 - \dfrac{\hat{\mu}_0}{\lambda_0 + \mu_0}\right)
 Q_{n+1}(0)\,
  \mathbb{E}_n[\tau_{n+1} \mid \tau_{n+1} < \tau_{-1}] 
  + \frac{\hat{\mu}_0}{
    \mu_0 (\lambda_0 + \mu_0)\lambda_n\, \mathbb{E}_n\left[T_n^{(-1)}\right]}\right]^{-1},\quad n\geq0.
\end{equation}
Finally, from the value of $\tilde q_n$ in \eqref{solqt} and using formula \eqref{expocctime}, we obtain, after some computations,
\begin{equation}\label{yynnprob}
\tilde{y}_n
= -\left[\left(1 - \dfrac{\hat{\mu}_0}{\lambda_0 + \mu_0}\right)
 Q_{n}(0)\,
  \mathbb{E}_n[\tau_{n+1} \mid \tau_{n+1} < \tau_{-1}] 
  + \frac{\hat{\mu}_0}{
    \mu_0 (\lambda_0 + \mu_0)\mu_{n+1}\, \mathbb{E}_{n+1}\left[T_{n+1}^{(-1)}\right]}\right]^{-1}, n\geq0.
\end{equation}

There are some situations (apart from the one described for the dual process in Remark \ref{remdual}) where the coefficients $\tilde{s}_n$, $\tilde{r}_n$, $\tilde{x}_n$ and $\tilde y_n$ simplify considerably or admit a different probabilistic interpretation.

\begin{itemize}[leftmargin=0.25in]
\item Case $\mu_0=0$ and $B<\infty$. In this setting, the birth--death process is conservative and positive recurrent. Let $X_\pi$ be a random variable with the stationary distribution, defined by
\[
\mathbb{P}(X_\pi=n)=\frac{\pi_n}{B},\quad n\geq0.
\]
Since $Q_n(0)=1$ (see~\eqref{Qn0}), the coefficient $\tilde{s}_n$ in~\eqref{ssnn} can be rewritten as
\begin{equation}\label{ssnnB}
\tilde{s}_n
=
\frac{1}{\mathbb{P}(X_\pi=n)}
\left[
\frac{\hat{\mu}_0}{\lambda_0B}
+
\left(1-\frac{\hat{\mu}_0}{\lambda_0}\right)
F_{X_\pi}(n)
\right],
\quad n\geq0,
\end{equation}
where $F_{X_\pi}(n)=\mathbb{P}(X_\pi\leq n)$ denotes the cumulative distribution function of $X_\pi$. The remaining coefficients can be computed directly from~\eqref{algrelLU}, taking into account that $\tilde q_n=-1$ for $n\geq0$ and $\tilde t_n=1-\hat\mu_0/\lambda_0$ for $n\geq1$. The absorption rate $\hat\mu_0$ of the Darboux-transformed process must satisfy condition~\eqref{condLU2}. If $\hat\mu_0=0$, the birth--death rates of $\widehat{\mathcal{A}}$ in~\eqref{newbdcoe} take the form
\begin{equation}\label{adfsdf}
\hat{\lambda}_n
=
\mu_{n+1}\frac{F_{X_\pi}(n+1)}{F_{X_\pi}(n)},
\quad
\hat{\mu}_{n+1}
=
\lambda_{n+1}\frac{F_{X_\pi}(n)}{F_{X_\pi}(n+1)},
\quad n\geq0.
\end{equation}
On the other hand, if $\hat\mu_0=\lambda_0B/(B-1)$, then
\[
\hat{\lambda}_n
=
\mu_{n+1}\frac{1-F_{X_\pi}(n+1)}{1-F_{X_\pi}(n)},
\quad
\hat{\mu}_{n+1}
=
\lambda_{n+1}\frac{1-F_{X_\pi}(n)}{1-F_{X_\pi}(n+1)},
\quad n\geq0.
\]

As noted in Remark~\ref{remdual}, the intermediate case $\hat{\mu}_0=\lambda_0$, for which $\hat{\lambda}_n=\mu_{n+1}$ and $\hat{\mu}_n=\lambda_n$ for $n\geq0$, corresponds to the \emph{dual process}. Its infinitesimal generator is denoted by $\mathcal{A}^d$, or equivalently by $\widehat{\mathcal{A}}^{(\lambda_0)}$ in the notation introduced at the beginning of Section~\ref{sec2}. This case lies between the two limiting values of $\hat{\mu}_0$ described above, since $B/(B-1)>1$.

It is not difficult to verify that the sequence $h(n)=F_{X_\pi}(n)$ is harmonic with respect to the dual process $\mathcal{A}^d$. Therefore, the operator corresponding to $\hat{\mu}_0=0$, namely $\widehat{\mathcal{A}}^{(0)}$, is precisely the \emph{Doob $h$-transform} (see~\eqref{doobt}) of $\mathcal{A}^d$ (or $\widehat{\mathcal{A}}^{(\lambda_0)}$), with $h(n)=F_{X_\pi}(n)$. Thus, the process associated with $\widehat{\mathcal{A}}^{(0)}$ may be interpreted as the dual process transformed through the harmonic function $F_{X_\pi}$. The other limiting case, corresponding to $\hat{\mu}_0=\lambda_0B/(B-1)$, is instead related to the tail function $h(n)=1-F_{X_\pi}(n)$. This function is harmonic only in the interior, since $(\mathcal{A}^dh)(n)=0$ for $n\geq1$, whereas $\mathcal{A}^dh=-\lambda_0e^{(0)}$ at the boundary. Hence, under the boundary convention adopted here, this case should not be regarded as a global Doob $h$-transform in the sense of~\eqref{doobt}, although the corresponding $h$-transform relation remains valid for the interior rates. More generally, while the coefficients $\tilde{s}_n$ in~\eqref{ssnnB} interpolate between these two limiting cases, the sequence $(\tilde{s}_n)_{n\geq0}$ is not, in general, harmonic with respect to $\mathcal{A}^d$.
\smallskip

\item Case $\mu_0=\hat\mu_0=0$. Here no stationary distribution is needed (for $B<\infty$ this is the case $\hat\mu_0=0$ above, written differently); the coefficient $\tilde s_n$ in \eqref{ssnn} admits a different probabilistic interpretation.  
In this situation we have $\tilde q_n=-1$ for $n\ge 0$ and
$\tilde t_n=1$ for $n\ge 1$.  
Hence, using \eqref{ssnn} together with \eqref{Expx}, we obtain
\[
\tilde{s}_n
 = \frac{1}{\pi_n}\sum_{k=0}^n \pi_k
 = \lambda_n\,\mathbb{E}_n[\tau_{n+1}],
 \quad n\ge 0.
\]
The remaining coefficients follow from \eqref{algrelLU}:
\[
\tilde r_n
 = 1 - \lambda_n\mathbb{E}_n[\tau_{n+1}]
 = -\,\mu_n\,\mathbb{E}_{n-1}[\tau_n],
 \quad n\ge 1,\quad 
\tilde x_n = -\tilde y_n
 = \frac{1}{\mathbb{E}_n[\tau_{n+1}]},
 \quad n\ge 0.
\]
The expression for $\tilde r_n$ follows from the identity
\(
\lambda_n\mathbb{E}_n[\tau_{n+1}]
 = 1 + \mu_n\mathbb{E}_{n-1}[\tau_n], n\ge 1.
\)
Therefore, using \eqref{newbdcoe}, the birth--death rates of the
Darboux-transformed process are
\[
\hat{\lambda}_n
 = \lambda_{n+1}
   \frac{\mathbb{E}_{n+1}[\tau_{n+2}]}
        {\mathbb{E}_n[\tau_{n+1}]},
\quad
\hat{\mu}_{n+1}
 = \mu_{n+1}
   \frac{\mathbb{E}_{n}[\tau_{n+1}]}
        {\mathbb{E}_{n+1}[\tau_{n+2}]},
\quad n\ge 0.
\]
\smallskip

\item Case $\mu_0>0$ and $\hat\mu_0=0$. In this case we start with a nonconservative birth--death process and construct a new conservative birth--death process by using the Darboux transformation. In this situation we have $\tilde q_n=-Q_{n+1}(0)/Q_n(0), n\geq0,$ and $\tilde t_n=Q_n(0), n\geq0$. Therefore, from 
\eqref{ssnnprob}--\eqref{yynnprob} and Lemma \ref{lem21}, we obtain directly
$$
\tilde s_n=\lambda_n\frac{\mathbb{E}_n[\tau_{n+1} \mid \tau_{n+1} < \tau_{-1}]}{\mathbb{P}_0(\tau_{n+1}<\tau_{-1})},\quad\tilde r_n=-\mu_n\frac{\mathbb{E}_{n-1}[\tau_{n} \mid \tau_{n} < \tau_{-1}]}{\mathbb{P}_0(\tau_{n-1}<\tau_{-1})},\quad n\geq1,
$$
$$
\tilde x_n=\dfrac{\mathbb{P}_0(\tau_{n+1}<\tau_{-1})}{\mathbb{E}_n[\tau_{n+1} \mid \tau_{n+1} < \tau_{-1}]},\quad\tilde y_n=-\dfrac{\mathbb{P}_0(\tau_{n}<\tau_{-1})}{\mathbb{E}_n[\tau_{n+1} \mid \tau_{n+1} < \tau_{-1}]},\quad n\geq0.
$$
Again, from \eqref{newbdcoe}, we obtain the birth--death rates for the Darboux-transformed process:
\begin{equation*}
\begin{split}
\hat{\lambda}_n&=\lambda_{n+1}\frac{\mathbb{E}_{n+1}[\tau_{n+2} \mid \tau_{n+2} < \tau_{-1}]}{\mathbb{E}_n[\tau_{n+1} \mid \tau_{n+1} < \tau_{-1}]}\frac{\mathbb{P}_0(\tau_{n+1}<\tau_{-1})}{\mathbb{P}_0(\tau_{n+2}<\tau_{-1})},\quad n\geq0,\\
\quad\hat{\mu}_{n+1}&=\mu_{n+1}\frac{\mathbb{E}_{n}[\tau_{n+1} \mid \tau_{n+1} < \tau_{-1}]}{\mathbb{E}_{n+1}[\tau_{n+2} \mid \tau_{n+2} < \tau_{-1}]}\frac{\mathbb{P}_0(\tau_{n+1}<\tau_{-1})}{\mathbb{P}_0(\tau_{n}<\tau_{-1})},\quad n\geq0.
\end{split}
\end{equation*}

\item Case $\mu_0>0$ and $\hat\mu_0=\lambda_0+\mu_0$. Now both birth--death processes are nonconservative. In this situation we have $\tilde q_n=-Q_{n+1}(0)/Q_n(0), n\geq0,$ and $\tilde t_n=0, n\geq1$. Therefore, from 
\eqref{ssnnprob}--\eqref{yynnprob}, we obtain directly
$$
\tilde s_n=\frac{1}{\mu_0\mathbb{E}_n\left[T_n^{(-1)}\right]},\quad n\geq0,\quad \tilde r_n=-\frac{1}{\mu_0\mathbb{E}_n\left[T_n^{(-1)}\right]},\quad n\geq1,
$$
and 
$$
\tilde x_n=\mu_0\lambda_n\mathbb{E}_n\left[T_n^{(-1)}\right],\quad\tilde y_n=-\mu_0\mu_{n+1}\mathbb{E}_{n+1}\left[T_{n+1}^{(-1)}\right],\quad n\geq0.
$$
Again, from \eqref{newbdcoe} and Lemma \ref{lem21}, we obtain the birth--death rates for the Darboux-transformed process:
$$
\hat{\lambda}_n=\lambda_{n}\frac{\mathbb{E}_n\left[T_n^{(-1)}\right]}{\mathbb{E}_{n+1}\left[T_{n+1}^{(-1)}\right]},\quad\hat{\mu}_{n+1}=\frac{\lambda_{n+1}\mu_{n+1}}{\lambda_{n}}\frac{\mathbb{P}_{n}(\tau_{n+1}<\tau_{-1})}{\mathbb{P}_{n+1}(\tau_{n+2}<\tau_{-1})}\frac{\mathbb{E}_{n+1}\left[T_{n+1}^{(-1)}\right]}{\mathbb{E}_{n}\left[T_{n}^{(-1)}\right]},\quad n\geq0.
$$
\end{itemize}

\begin{remark}
The probabilistic interpretation of the coefficients arising in the LU factorization is another new feature with respect to the discrete-time case. Although in that setting the coefficients were themselves transition probabilities, no connections with hitting probabilities or expected stopping times were established in \cite{GdI}. 
\end{remark}

\section{UL factorization and the Darboux transformation}\label{Sec4}

The UL factorization may be regarded as the dual case of the LU factorization studied in Section~\ref{Sec3}. An important difference, however, is that the UL factorization is parameterized by a \emph{free variable}, denoted by $x_0$, in terms of which all the remaining coefficients 
are subsequently determined. Consider now $\mathcal{A}=\mathcal{A}_U\mathcal{A}_L$, where

\begin{equation}\label{AL_AU2}
	\begin{array}{cc}
		\mathcal{A}_U=
		\begin{pmatrix}
			y_0 & x_0 &  &  & \\
			 &y_1 & x_1 &  & \\
			&  &y_2 &  x_2 & \\
			 &  &  & \ddots & \ddots
		\end{pmatrix},\quad 
		\mathcal{A}_L=
		\begin{pmatrix}
			s_0 &  &  &  \\
			r_1 & s_1 &  &  \\
			 & r_2 &s_2 &  \\
			 & & \ddots    & \ddots
		\end{pmatrix}.
	\end{array}
\end{equation}
This gives the system of equations
\begin{equation}\label{AL_AU_eq2}
\begin{split}
\lambda_n & = x_ns_{n+1},\quad n\geq0, \\	
-(\mu_n+\lambda_n) & = y_ns_n+x_nr_{n+1}, \quad n\geq0,\\	
\mu_n &= r_ny_n, \quad n\geq1.
\end{split}
\end{equation}
Since the birth and death rates $\{\lambda_n,\mu_n\}$ are positive, we have that $s_{n+1}x_n>0, n\geq 0,$ and $r_ny_{n}>0, n\geq 1$; hence, $s_{n+1},x_n\neq 0, n\geq 0,$ and $r_n,y_{n}\neq 0, n\geq 1$. Additionally, we are interested in the case where the Darboux transformation $\widetilde{\mathcal{A}}=\mathcal{A}_L\mathcal{A}_U $ produces a new birth--death process. That means that we need to have that $\widetilde{\mathcal{A}}\bm e=-\tilde{\mu}_0e^{(0)}$, i.e.
\begin{equation}\label{AU_AL_eq2}
\begin{split}
\tilde{\lambda}_n & =s_nx_n, \quad n\geq0, \\
-(\tilde{\lambda}_n+\tilde{\mu}_n) & = r_nx_{n-1}+s_ny_n,\quad n\geq 1,\quad-(\tilde\lambda_0+\tilde\mu_0)=s_0y_0,\\
\tilde{\mu}_n &= r_ny_{n-1}, \quad n\geq1,
\end{split}
\end{equation}
with $\tilde\lambda_n,\tilde\mu_{n+1}>0, n\geq0,$ and $\tilde\mu_0\geq0$. 
\smallskip

As before, we may assume without loss of generality that $s_0=1$. Now, defining $\tilde{Q}=\mathcal{A}_LQ$, we have
\begin{equation}\label{QQbb1}
\tilde Q_0(x)=s_0Q_0(x),\quad\tilde{Q}_n(x) = r_n Q_{n-1}(x) + s_n Q_n(x), \quad n \geq 1.
\end{equation}
These polynomials have degree $n$ and satisfy the eigenvalue equation $-x \tilde{Q} = \widetilde{\mathcal{A}}\tilde{Q}$. From~\eqref{QQbb1} we see that $\tilde{Q}_0(x)= s_0$. Therefore, the sequence $(\tilde{Q}_n)_{n\in\mathbb{N}_0}$ will correspond to the birth--death polynomials associated with $\widetilde{\mathcal{A}}$ if and only if $s_0 = 1$. 

As for the spectral measure associated with $\widetilde{\mathcal{A}}$, it is well known (see \cite{GdI,Yo}) that the polynomials $(\tilde Q_n)_{n\in\mathbb{N}_0}$ are orthogonal with respect to the~\emph{Geronimus transform} of the measure $d\psi$, given by
\begin{equation}\label{spmec2e}
d\widetilde{\psi}(x)
=\frac{x_0+\tilde\mu_0}{x}\,d\psi(x)+M\,\delta_0(dx),
\quad
M=1-(x_0+\tilde\mu_0)m_{-1},
\end{equation}
where \(\delta_0\) denotes the Dirac measure at \(0\) and $m_{-1} = \int_{0}^{\infty} x^{-1} d\psi(x)$, assuming implicitly that $m_{-1} < \infty$. Here $M$ makes $d\widetilde\psi$ a probability measure, while the factor $x_0+\tilde\mu_0$ is forced by the orthogonality of $\tilde Q_1(x)=(x_0+\tilde\mu_0-x)/x_0$ (see \eqref{QQbb1}) to $\tilde Q_0=1$. Observe also that, from \eqref{AU_AL_eq2}, we have $y_0=-(x_0+\tilde\mu_0)$. From \cite[(3.70)]{MDIB} we have that
\begin{equation}\label{mmm1}
m_{-1}=\frac{A}{1+\mu_0A}=\left(\mu_0+\frac{1}{A}\right)^{-1},
\end{equation}
where $A$ is defined in \eqref{quantity_AB}. Therefore, if $\mu_0>0$, we always have that $m_{-1} < \infty$, but if $\mu_0=0$ then $m_{-1}$ is finite only when $A<\infty$, that is, the original birth--death process is transient. The potential coefficients $(\tilde\pi_n)_{n\in\mathbb{N}_0}$ of $\widetilde{\mathcal{A}}$ are given by
\begin{equation*}
\tilde{\pi}_0=1, \quad\tilde{\pi}_n=\frac{\tilde{\lambda}_0\cdots\tilde{\lambda}_{n-1}}{\tilde{\mu}_1\cdots\tilde{\mu}_n}=\frac{y_n}{y_0s_n}\pi_n,  \quad n\geq 1.
\end{equation*}
The construction of the spectral measure and the orthogonal polynomials associated with the Darboux transformation is similar to that carried out in the discrete-time case \cite{GdI} (see also \cite{Yo}).

\medskip

As in the case of the LU factorization, we will give a procedure to compute the coefficients $x_n,y_n,s_n,r_n,$ recursively, in this order. For that, introduce the sequences (observe the difference with respect to \eqref{ttqq})
\begin{equation*}
t_0=-\tilde\mu_0,\;  t_n= x_n+y_n,\quad n\geq1,\quad q_n = \frac{r_n}{s_n},\quad n\geq1.
\end{equation*}
Summing the three equations in \eqref{AU_AL_eq2} and taking into account that $s_0=1$, we obtain the coupling conditions:
\begin{equation}\label{coucon2}
\begin{split}
&x_0+y_0=-\tilde\mu_0,\\
&s_n(x_{n}+y_{n})+r_n(x_{n-1}+y_{n-1})=0,\quad n\ge1.
\end{split}
\end{equation}
Observe that the value $t_0=x_0+y_0=-\tilde\mu_0$ is determined by the first of the equations above. Also, we will need an explicit expression of $q_1=r_1/s_1$. Since $s_1=\lambda_0/x_0$ (first equation in  \eqref{AL_AU_eq2} for $n=0$) and $r_1=-(\lambda_0+\mu_0-x_0-\tilde\mu_0)/x_0$ (which follows from the second equation in \eqref{AL_AU_eq2} for $n=0$ together with $y_0=-(x_0+\tilde\mu_0)$) we have that (see \eqref{Qn0})
\begin{equation}\label{qq1}
q_1=\frac{r_1}{s_1}=-1-\frac{\mu_0}{\lambda_0}+\frac{x_0+\tilde\mu_0}{\lambda_0}=-Q_1(0)+\frac{x_0+\tilde\mu_0}{\lambda_0}.
\end{equation}
We may write the coupling conditions \eqref{coucon2} as the first-order difference equation
\begin{equation}\label{deqtt2}
t_{n} = -q_n t_{n-1},\quad n\geq1,\quad t_0=-\tilde\mu_0.
\end{equation}
On the other hand, substituting \(y_n=\dfrac{\mu_n}{r_{n}}\) and \(x_n=\dfrac{\lambda_n}{s_{n+1}}\) into the second relation in \eqref{AL_AU_eq2}, gives
\[
-(\lambda_n+\mu_n)
= \frac{\mu_n}{q_{n}} + \lambda_n q_{n+1}.
\]
Rearranging this identity we get the nonlinear first-order recurrence equation
\begin{equation}\label{deqqq2}
q_{n+1}= -1 - \frac{\mu_n}{\lambda_n} - \frac{\mu_n}{\lambda_nq_{n}},\quad n\ge1,\quad q_1 = -1-\frac{\mu_0}{\lambda_0}+\frac{x_0+\tilde\mu_0}{\lambda_0}.
\end{equation}
Notice the similarity of \eqref{deqtt2} and \eqref{deqqq2} with the recurrence relations \eqref{deqtt} and \eqref{deqqq}.

\begin{proposition}
The explicit solutions of the recurrence equations \eqref{deqtt2} and \eqref{deqqq2} are given by
\begin{equation}\label{solqt2}
\begin{split}
q_n&= -\frac{Q_{n}(0)+(x_0+\tilde\mu_0)Q_{n}^{(0)}(0)}{Q_{n-1}(0)+(x_0+\tilde\mu_0)Q_{n-1}^{(0)}(0)},\quad n\geq1, \\
t_n &= -\tilde\mu_0\left(Q_{n}(0)+(x_0+\tilde\mu_0)Q_{n}^{(0)}(0)\right),\quad n\geq0,
\end{split}
\end{equation}
where $Q_n(0)$ and $Q_n^{(0)}(0)$ (see \eqref{Qn0} and \eqref{Qn00}) are two linearly
independent solutions, for $n\ge1$, of the three-term recurrence \eqref{QQs} at $x=0$,
determined by the initial conditions $Q_0=1$, $Q_1(0)=1+\mu_0/\lambda_0$ and \eqref{iccsqn0}, respectively.
\end{proposition}
\begin{proof}
If we write \(q_n=-v_n/v_{n-1}\), \(n\geq1\), in \eqref{deqqq2}, then the sequence \(v_n\) satisfies the interior harmonic relation
\[
\lambda_nv_{n+1}-(\lambda_n+\mu_n)v_n+\mu_nv_{n-1}=0,\quad n\geq1,
\]
with \(v_1=-q_1v_0\), where $v_0\neq0$ is an arbitrary normalization constant. This is precisely the recurrence \eqref{QQs} evaluated at \(x=0\), but with different initial conditions. Hence, its general solution is a linear combination of the two linearly independent solutions \(Q_n(0)\) and \(Q_n^{(0)}(0)\). If we write \(v_n=a\,Q_n(0)+b\,Q_n^{(0)}(0)\), then, on the one hand, for $n=0$, we have \(a=v_0\), since \(Q_0^{(0)}(0)=0\). On the other hand, for $n=1$, using that \(v_1=-q_1v_0\) and \(Q_1^{(0)}(x)=-1/\lambda_0\) (see \eqref{iccsqn0}), we obtain \(b=\lambda_0v_0(q_1+Q_1(0))\). Hence, we obtain
$$
v_n=v_0\left(Q_n(0)+\lambda_0(q_1+Q_1(0))Q_n^{(0)}(0)\right),\quad n\geq0.
$$
This solution can be written, using \eqref{qq1}, as
\begin{equation}\label{uun2}
v_n=v_0\left(Q_n(0)+(x_0+\tilde\mu_0)Q_n^{(0)}(0)\right),\quad n\geq0.
\end{equation}
Therefore we get the solution for $q_n$ by definition. On the other hand, solving directly \eqref{deqtt2}, we obtain
\[
t_n=-\tilde\mu_0\prod_{k=1}^{n}(-q_k), \quad n\ge0,
\]
with the convention that $\prod_{k=1}^{0}=1$. Using the representation for $q_n$, the initial value $t_0$ in \eqref{deqtt2} and \eqref{uun2}, the product telescopes and gives the closed form
$$
t_n =-\tilde\mu_0\frac{v_n}{v_0}=-\tilde\mu_0\left(Q_n(0)+(x_0+\tilde\mu_0)Q_n^{(0)}(0)\right),\quad n\geq0,
$$
obtaining the result.
\end{proof}
If the sequences $t_n$ and $q_n$ are known, then
\begin{equation*}
y_n=\frac{\mu_n}{r_n}=\frac{\mu_n}{q_ns_n}=\frac{\mu_nx_{n-1}}{\lambda_{n-1}q_n}\quad \Longrightarrow\quad
x_n=t_{n}-\frac{\mu_n}{\lambda_{n-1}q_n}x_{n-1},\quad n\geq1.
\end{equation*}
This gives a recursive way of obtaining explicitly $x_n$ with $x_0$ a free parameter. Indeed,
$$
x_n =-\tilde\mu_0\left(Q_n(0)+(x_0+\tilde\mu_0)Q_n^{(0)}(0)\right)+ \frac{\mu_n }{\lambda_{n-1}}\frac{Q_{n-1}(0)+(x_0+\tilde\mu_0)Q_{n-1}^{(0)}(0)}{Q_n(0)+(x_0+\tilde\mu_0)Q_n^{(0)}(0)}x_{n-1},\quad n\geq1.
$$
A straightforward computation using the definition of the potential coefficients $(\pi_n)_{n\in\mathbb{N}_0}$ in \eqref{potcoef} gives
\begin{equation}\label{xxnn2}
x_n=\frac{1}{\pi_nu_n}\left[x_0+\tilde\mu_0 -\tilde\mu_0\sum_{k=0}^n \pi_ku_k^2\right], \quad n\ge0,
\end{equation}
where $x_0$ is a \emph{free parameter} and 
\begin{equation}\label{uun}
u_n=Q_n(0)+(x_0+\tilde\mu_0)Q_n^{(0)}(0),\quad n\geq0.
\end{equation}
Once we have $x_n$ and using the initial conditions, the rest of sequences $y_n,r_n,s_n$ are obtained by the following algebraic relations:
\begin{equation}\label{algrelUL}
y_n=t_n-x_n,\quad n\geq0,\quad s_{n}=\frac{\lambda_{n-1}}{x_{n-1}},\quad r_n=q_ns_n=\frac{\lambda_{n-1}q_n}{x_{n-1}},\quad n\geq1.
\end{equation}
In other words,
\begin{align}
\label{yynn2}y_n&=-\frac{1}{\pi_nu_n}\left[x_0 +\tilde\mu_0-\tilde\mu_0\sum_{k=0}^{n-1} \pi_ku_k^2\right],\quad n\ge1,\quad y_0=-(x_0+\tilde\mu_0),\\
\label{ssnn2}s_n&=\frac{\lambda_{n-1}\pi_{n-1}u_{n-1}}{x_0 +\tilde\mu_0-\tilde\mu_0\displaystyle\sum_{k=0}^{n-1} \pi_ku_k^2},\quad n\ge1,\quad s_0=1,\\
\label{rrnn2}r_n&=-\frac{\mu_{n}\pi_{n}u_{n}}{x_0+\tilde\mu_0 -\tilde\mu_0\displaystyle\sum_{k=0}^{n-1} \pi_ku_k^2},\quad n\ge1.
\end{align}
From \eqref{AU_AL_eq2}, \eqref{algrelUL}, \eqref{solqt2}, and \eqref{uun}, the birth--death rates of the Darboux transformation $\widetilde{\mathcal{A}}$ can be written as
\begin{equation}\label{newbdcoe2}
\tilde\lambda_n
=
\lambda_{n-1}\frac{x_n}{x_{n-1}},
\quad
\tilde\mu_{n+1}
=
\frac{\lambda_n\mu_n}{\lambda_{n-1}}
\frac{u_{n+1}u_{n-1}}{u_n^2}
\frac{x_{n-1}}{x_n},\quad n\geq1.
\end{equation}
The boundary case $n=0$ requires some care. If $\mu_0>0$, the same formulas may be extended to $n=0$ by introducing the conventions $\lambda_{-1}=x_{-1}$ and $u_{-1}=(x_0+\tilde\mu_0)/\mu_0$, the latter following from $Q_{-1}^{(0)}(0)=1/\mu_0$. When $\mu_0=0$, this convention for $u_{-1}$ is not defined, so the boundary rates must instead be obtained directly from the factorization identities. In either case, $s_0=1$ gives $\tilde\lambda_0=x_0$. Moreover, evaluating \eqref{QQs} at $n=0$ and $x=0$ gives $Q_1(0)=(\lambda_0+\mu_0)/\lambda_0$, while \eqref{iccsqn0} gives $Q_1^{(0)}(0)=-1/\lambda_0$. Hence, from \eqref{uun},
$u_1=(\lambda_0+\mu_0-x_0-\tilde\mu_0)/\lambda_0$, which reduces to $u_1=(\lambda_0-x_0-\tilde\mu_0)/\lambda_0$ when $\mu_0=0$. Since $s_1=\lambda_0/x_0$, $q_1=r_1/s_1=-u_1$, and $y_0=-(x_0+\tilde\mu_0)$, we obtain
\begin{equation}\label{laconcha3r}
\tilde\lambda_0=x_0,
\quad
\tilde\mu_1
=
\frac{(\lambda_0+\mu_0-x_0-\tilde\mu_0)(x_0+\tilde\mu_0)}{x_0}.
\end{equation}
Thus, the boundary formulas in \eqref{laconcha3r} are valid for both $\mu_0=0$ and $\mu_0>0$, whereas \eqref{newbdcoe2} applies directly for $n\geq1$.

\begin{remark}
As in the LU factorization, formulas \eqref{Qn0} and \eqref{Qn00} for \(Q_n(0)\) and \(Q_n^{(0)}(0)\), respectively, provide all the coefficients \(r_n,s_n,x_n,y_n\) explicitly in terms of the birth--death rates $\{\lambda_n,\mu_n\}$ and the potential coefficients $\pi_n$ in \eqref{potcoef}. Again, this improves the discrete-time case \cite{GdI}, where the factor coefficients were only obtained recursively, rather than explicitly as done here. 
\end{remark}

\begin{remark}\label{remdual2}
Observe that if we choose $\mu_0>0$ and $\tilde\mu_0=0$ with the choice of the free parameter $x_0$ as $x_0=\mu_0$, we obtain a simplified UL factorization with coefficients:
$$
s_0=1,\quad s_{n+1}=\frac{\lambda_{n}\pi_{n}}{\mu_0},\quad r_{n+1}=-\frac{\lambda_{n}\pi_{n}}{\mu_0},\quad x_n=\frac{\mu_0}{\pi_n},\quad y_n=-\frac{\mu_0}{\pi_n},\quad n\geq0.
$$
The birth--death rates of the Darboux transformation $\widetilde{\mathcal{A}}$ (or $\widetilde{\mathcal{A}}^{(0)}$
in the notation used at the beginning of Section~\ref{sec2}) are given by (see \eqref{newbdcoe2} and \eqref{laconcha3r}) 
$
\tilde\lambda_n=\mu_{n},\tilde\mu_{n+1}=\lambda_{n},n\geq0.
$
This factorization is closely related to the family of \emph{dual polynomials} $(H_n)_{n\in\mathbb{N}_0}$ already introduced in \cite{KMc2} (for the case $\mu_0>0$). It is easy to see that in our case, these dual polynomials coincide with the family $(\tilde Q_n)_{n\in\mathbb{N}_0}$ defined by \eqref{QQbb1}. For more information, see \cite[Remark 3.12]{MDIB}, where $\widetilde{\mathcal{A}}^{(0)}$ is denoted by $\mathcal{A}^d$.
\end{remark}
\begin{remark}
The family of orthogonal polynomials $(\tilde{Q}_n)_{n\in\mathbb{N}_0}$ associated with the Darboux transformation $\widetilde{\mathcal{A}}$, and orthogonal with respect to the Geronimus transform of $d\psi$ given in~\eqref{spmec2e}, can be simplified by means of the Christoffel--Darboux formulas~\eqref{CDF} and \eqref{CDF2}. From \eqref{QQbb1}, the formula $r_n =q_ns_n$, and the definition of $q_n$ in~\eqref{solqt2} (see also \eqref{uun}), we obtain \(\tilde Q_0=1\) and, for $n\geq1$,
\begin{align}
\tilde{Q}_n(x)
\nonumber&= s_n \left( Q_{n}(x) - \frac{u_n}{u_{n-1}} Q_{n-1}(x) \right)
= \frac{s_n}{u_{n-1}} \left(u_{n-1}Q_n(x)-u_nQ_{n-1}(x)\right) \\[0.3em]
\nonumber&= \frac{s_n}{u_{n-1}} \left[Q_n(x)Q_{n-1}(0)-Q_n(0)Q_{n-1}(x)+(x_0+\tilde\mu_0)\left(Q_n(x)Q_{n-1}^{(0)}(0)-Q_n^{(0)}(0)Q_{n-1}(x)\right)\right] \\[0.3em]
\label{QDTUL}&=\frac{s_n}{\lambda_{n-1}\pi_{n-1}u_{n-1}}\left(x_0+\tilde\mu_0-x\sum_{k=0}^{n-1}\pi_ku_kQ_k(x)\right)= \frac{x_0+\tilde\mu_0-x\D\sum_{k=0}^{n-1}\pi_ku_kQ_k(x)}{x_0+\tilde\mu_0-\tilde\mu_0\D\sum_{k=0}^{n-1}\pi_ku_k^2}.
\end{align}
The last two steps follow from the Christoffel--Darboux formulas~\eqref{CDF} and \eqref{CDF2} evaluated at $y = 0$, together with the definition of $s_n$ given in~\eqref{ssnn2}.
\end{remark}

Up to now, we have computed the coefficients $(x_n)_{n\in\mathbb{N}_0}$, $(y_n)_{n\in\mathbb{N}_0}$, $(s_n)_{n\in\mathbb{N}_0}$, and $(r_n)_{n\in\mathbb{N}}$ given by \eqref{xxnn2} and \eqref{yynn2}--\eqref{rrnn2} so that $\widetilde{\mathcal{A}}\bm e=-\tilde{\mu}_0e^{(0)}$, that is, the sum of every row of the Darboux transform \(\widetilde{\mathcal{A}}\) equals 0 (except for the first one, see \eqref{AU_AL_eq2}). 
Now, we are ready to state the second of the main results of this paper, which provides conditions ensuring that the Darboux transform \(\widetilde{\mathcal{A}}\) defines a genuine birth--death process. Throughout, we use the convention \(1/\infty=0\).

\begin{theorem}\label{theom37}
Let $\{X_t : t \ge 0\}$ be a birth--death process with birth and death rates 
$\{\lambda_n,\mu_n\}$ and infinitesimal generator $\mathcal{A}$ given by~\eqref{QQmm}. Assume the standing condition \eqref{uniqueness_assumption}.
Consider the UL factorization $\mathcal{A}=\mathcal{A}_U\mathcal{A}_L$, where $\mathcal{A}_U$ and $\mathcal{A}_L$ are defined in~\eqref{AL_AU2}, 
with coefficients $(x_n)_{n\in\mathbb{N}_0}$, $(y_n)_{n\in\mathbb{N}_0}$, 
$(s_n)_{n\in\mathbb{N}_0}$, and $(r_n)_{n\in\mathbb{N}}$ given by \eqref{xxnn2} and \eqref{yynn2}--\eqref{rrnn2}, respectively.  
Let \(\widetilde{\mathcal{A}}=\mathcal{A}_L \mathcal{A}_U\) be the Darboux transform of $\mathcal{A}$ (see \eqref{AU_AL_eq2}) and define the series $T$ (see \eqref{uun}, \eqref{Qn0} and \eqref{Qn00}) as
\begin{equation}\label{TTT}
T=\sum_{n=0}^\infty\pi_nu_n^2=\sum_{n=0}^\infty\pi_n\left(1+(\mu_0-x_0-\tilde\mu_0)\sum_{k=0}^{n-1}\frac{1}{\lambda_k\pi_k}\right)^2.
\end{equation}
Then $\widetilde{\mathcal A}$ is the generator of a new birth--death process $\{\widetilde{X}_t : t \ge 0\}$ with birth and death rates $\{\tilde\lambda_n,\tilde\mu_n\}$ given by
\eqref{newbdcoe2}, \eqref{laconcha3r} if and only if
\begin{equation}\label{condULgen}
0<x_0+\tilde\mu_0\leq \mu_0+\frac{1}{A},
\quad
\tilde\mu_0\geq0,\quad\mbox{together with}\quad \begin{cases}
\tilde\mu_0 T\leq x_0+\tilde\mu_0, & T<\infty,\\[1mm]
\tilde\mu_0=0, & T=\infty.
\end{cases}
\end{equation}
Under these conditions, the admissible parameters split into the following two cases:
\begin{enumerate}
\item If $x_0+\tilde\mu_0\neq\mu_0+\frac{1}{A},$ then $T=\infty$. In this situation $\tilde\mu_0=0,$ and the free parameter $x_0$ is subject to the following condition:
\begin{equation}\label{cond1UL}
0<x_0<\mu_0+\frac{1}{A},
\end{equation}
where $A$ is defined by \eqref{quantity_AB}.
\item If $x_0+\tilde\mu_0=\mu_0+\frac{1}{A},$ (\(x_0\) is then determined by \(\tilde\mu_0\)) then $T$ is given by
\begin{equation}\label{TTT3}
T=\begin{cases}
\displaystyle
\sum_{n=0}^\infty \pi_n
\left(1-\frac{1}{A}\sum_{k=0}^{n-1}\frac{1}{\lambda_k\pi_k}\right)^2,
& \mbox{if } A<\infty\\[1ex]
B,
& \mbox{if } A=\infty
\end{cases}
\end{equation}
where $A$ and $B$ are defined by \eqref{quantity_AB}. In this situation the absorption rate $\tilde\mu_0\geq0$ of the Darboux-transformed process \(\widetilde{\mathcal{A}}\) must satisfy
\begin{equation}\label{condBAB}
0 \le \tilde{\mu}_0 \le \frac{1}{T}\left(\mu_0+\frac{1}{A}\right).
\end{equation}
In the case where $T=\infty,$ then $\tilde\mu_0=0$ and $x_0=\mu_0+\frac{1}{A}$.
\end{enumerate}
\end{theorem}
\begin{proof}
By the computation of the coefficients $(x_n)_{n\in\mathbb{N}_0}$, $(y_n)_{n\in\mathbb{N}_0}$, $(s_n)_{n\in\mathbb{N}_0}$, and $(r_n)_{n\in\mathbb{N}}$ given by \eqref{xxnn2} and \eqref{yynn2}--\eqref{rrnn2}, we have that $\widetilde{\mathcal{A}}\bm e=-\tilde{\mu}_0e^{(0)}$. \(\widetilde{\mathcal{A}}\) defines a new birth--death process if and only if \(\tilde{\lambda}_n, \tilde{\mu}_{n+1} > 0\) for all \(n \ge 0\) and \(\tilde{\mu}_0 \ge 0\). From \eqref{newbdcoe2} and \eqref{laconcha3r} it follows that $\tilde\lambda_n, \tilde\mu_{n+1} > 0$ for all $n \ge 0$ if and only if $u_n > 0$ and $x_n > 0$ for all $n \ge 0$. Indeed, recall first that
\(u_0=1,\) \(y_0=-(x_0+\tilde\mu_0),\) and, by construction,
\(\tilde\lambda_0=x_0.\)
Moreover, since \(s_1=\lambda_0/x_0\) and \(q_1=r_1/s_1=-u_1/u_0=-u_1,\) we obtain \(r_1=-\lambda_0u_1/x_0.\) Hence
\begin{equation}\label{tmuboundary}
\tilde\mu_1
=
r_1y_0
=
\lambda_0
\frac{x_0+\tilde\mu_0}{x_0}\,u_1.
\end{equation}
Suppose first that $x_n > 0$ and $u_n > 0$ for all $n \ge 0$. Since $\lambda_n,\mu_n>0$, it follows immediately from
\eqref{newbdcoe2} that \(\tilde\lambda_n>0,\tilde\mu_{n+1}>0,\) for all \(n\geq1.\) At the boundary, \(\widetilde\lambda_0=x_0>0,\)
while \eqref{tmuboundary}, together with
$\tilde\mu_0\geq0$, gives
\(\tilde\mu_1>0. \) Thus all the transformed birth and death rates are strictly positive.

Conversely, assume that \(\tilde\lambda_n>0,\) and \(
\tilde\mu_{n+1}>0,\) for all \(n\geq0.\) Since $\tilde\lambda_0=x_0$, we first obtain
\(x_0>0. \) Using successively the first identity in \eqref{newbdcoe2}, and the fact that $\lambda_{n-1}>0$, we obtain \(x_n/x_{n-1}>0, n\geq1.\)
Starting from $x_0>0$, this yields inductively
\(x_n>0,\) for \(n\geq0.\) It remains to prove the positivity of $u_n$. Since
\(u_0=1>0,\) and all the factors in \eqref{tmuboundary} except $u_1$ are strictly positive, the condition $\tilde\mu_1>0$ implies
\(u_1>0.\) Now, for every $n\geq1$, all factors in the second identity of
\eqref{newbdcoe2}, except $u_{n+1}u_{n-1}$, are strictly positive.
Therefore, \(\tilde\mu_{n+1}>0\) implies that \(u_{n+1}u_{n-1}>0.
\) Since $u_0>0$ and $u_1>0$, this relation yields recursively
\(u_n>0, n\geq0.\)

Now, on one hand, the condition $u_n>0$ implies, using \eqref{uun}, \eqref{Qn0} and \eqref{Qn00}, that
$$
x_0+\tilde\mu_0<\mu_0+\left(\sum_{k=0}^{n-1}\frac{1}{\lambda_k\pi_k}\right)^{-1},\quad n\geq1.
$$
Since the sum is a positive increasing sequence, the inverse is positive decreasing. Therefore
\begin{equation}\label{cond1ULp}
0<x_0+\tilde\mu_0\leq\mu_0+\frac{1}{A},
\end{equation}
where $A$ is defined by \eqref{quantity_AB}. It is worth noting that, when condition \eqref{cond1ULp} holds, the jump \( M \) of the discrete part of the measure \(d\widetilde{\psi}\) in \eqref{spmec2e} is nonnegative. 
This follows directly from \eqref{spmec2e} and \eqref{mmm1}, since
$M=1-(x_0+\tilde\mu_0)\big(\mu_0+\frac1A\big)^{-1}$. On the other hand, the condition $x_n>0$ for all $n\geq0$ in \eqref{xxnn2} implies, using $u_n>0$, that
$$
\tilde\mu_0<x_0\left(\sum_{k=1}^n\pi_ku_k^2\right)^{-1},\quad n\geq1.
$$
The sum above is a positive increasing sequence, so the inverse is positive and decreasing. Therefore
\begin{equation}\label{cond2UL}
\tilde\mu_0\leq \frac{x_0}{T-1},\quad\Longrightarrow\quad \tilde\mu_0 T\leq x_0+\tilde\mu_0,
\end{equation}
where $T$ is defined by \eqref{TTT}. Observe that $T>1$ since $u_0=1$ and $\pi_0=1$. 
\smallskip

Now, let us first assume that $x_0+\tilde\mu_0=\mu_0+\frac{1}{A}$. Substituting this identity into \eqref{TTT} immediately yields \eqref{TTT3} when $A<\infty$. On the other hand, if $A=\infty$, then, since $x_0+\tilde\mu_0=\mu_0$, we obtain $u_n\equiv1$. Consequently, we get $T=B$. Since condition \eqref{cond1ULp} is automatically satisfied, condition \eqref{cond2UL} reduces to \eqref{condBAB}. Moreover, if $T=\infty$, then \eqref{condBAB} immediately implies that $\tilde\mu_0=0$.

Let us see now the case where $x_0+\tilde\mu_0\neq\mu_0+\frac{1}{A}$. Assume first that $\mu_0>0$. By the standing assumption \eqref{uniqueness_assumption}, we have \(S=\infty\). In particular, this implies that $A+B=\infty$. Indeed, if both $A<\infty$ and $B<\infty$, then \eqref{Qn0} yields $1\leq Q_n(0)\leq 1+\mu_0A<\infty$, from which it would follow that $S<\infty$, a contradiction. We may distinguish three possible situations:
\begin{itemize}[leftmargin=0.25in]
\item $A<\infty$ and $B=\infty$. Since $A<\infty$, the sequence $(u_n)_{n\in\mathbb{N}_0}$ converges, with limit $u_\infty=\lim_{n\to\infty}u_n=1+(\mu_0-x_0-\tilde{\mu}_0)A$. Since $x_0+\tilde\mu_0\neq\mu_0+\frac{1}{A}$, it follows that $u_\infty\neq0$. As $B=\infty$, we immediately have that $T=\infty$.
\item $A=\infty$ and $B<\infty$. In this case, the sequence $(u_n)_{n\in\mathbb{N}_0}$ diverges. As $n\to\infty$, we have $u_n^2\sim(\mu_0-x_0-\tilde{\mu}_0)^2a_n^2$, where $a_n=\sum_{k=0}^{n-1}(\lambda_k\pi_k)^{-1}\to\infty$. Moreover, since $Q_n^2(0)\sim \mu_0^2a_n^2$ as $n\to\infty$ and $S=\infty$ by \eqref{SSS}, it follows that $T=\infty$.
\item $A=\infty$ and $B=\infty$. Then \eqref{cond1ULp} and
$x_0+\tilde\mu_0\neq\mu_0+\frac1A=\mu_0$ give $x_0+\tilde\mu_0<\mu_0$, so that
$u_n\ge1$ for all $n$ and hence $T\ge B=\infty$.
\end{itemize}
Therefore, in all three cases above, we have $T=\infty$. Consequently, \eqref{cond2UL} implies that $\tilde\mu_0=0$, while \eqref{cond1ULp} yields \eqref{cond1UL}. Let us now consider the case $\mu_0=0$. In this situation, one necessarily has $A<\infty$. Indeed, if $A=\infty$, then condition~\eqref{cond1ULp} cannot be satisfied, since we are assuming $x_0>0$. Therefore, since $A+B=\infty$ by the standing assumption \eqref{uniqueness_assumption}, it follows that $B=\infty$. Moreover, since $x_0+\tilde{\mu}_0\neq \frac{1}{A}$, we have $u_\infty\neq0$, and consequently $T=\infty$. Hence, \eqref{cond2UL} implies that $\tilde\mu_0=0$, while \eqref{cond1ULp} yields \eqref{cond1UL} with $\mu_0=0$.

For the converse implication, assume \eqref{condULgen}, with \(\tilde\mu_0\geq0\), and set \(a_n=\sum_{k=0}^{n-1}(\lambda_k\pi_k)^{-1}\uparrow A\). If \(x_0+\tilde\mu_0\leq\mu_0\), then \(u_n=1+(\mu_0-x_0-\tilde\mu_0)a_n\geq1\). Otherwise, \eqref{condULgen} implies \(A<\infty\) and \(0<x_0+\tilde\mu_0-\mu_0\leq1/A\); since \(a_n<A\) for every finite \(n\), we obtain \(u_n=1-(x_0+\tilde\mu_0-\mu_0)a_n>0\). Hence \(u_n>0\) for all \(n\geq0\). On the other hand, if \(\tilde\mu_0=0\), then \eqref{xxnn2} gives \(x_n=x_0/(\pi_nu_n)>0\). If \(\tilde\mu_0>0\), then \(T<\infty\) and \eqref{condULgen} yields \(x_0\geq\tilde\mu_0(T-1)\). Since \(\sum_{k=1}^n\pi_ku_k^2<T-1\) for every finite \(n\), it follows from \eqref{xxnn2} that \(x_n>0\). Therefore \eqref{newbdcoe2}, \eqref{laconcha3r}, and \eqref{tmuboundary} imply \(\tilde\lambda_n>0\) and \(\tilde\mu_{n+1}>0\) for all \(n\geq0\). Together with \(\tilde\mu_0\geq0\), this proves that \(\widetilde{\mathcal A}\) is the generator of a birth--death process.
\end{proof}

\begin{remark}
A geometric interpretation of the series $T$ in \eqref{TTT3} is obtained by observing that $T=\|q\|_{\ell_\pi^2(\mathbb{N}_0)}^2$, where $q(n)=\mathbb{P}_n(\tau_0<\infty)$ (see \eqref{qqprob}). In other words, $T$ is the squared norm of the hitting-probability sequence $q=(q(n))_{n\in\mathbb{N}_0}$ in the Hilbert space $\ell_\pi^2(\mathbb{N}_0)$. Observe that in general $0\le q(n)\le1$ for all $n\ge0$, but if $A=\infty$ (that is, if the birth--death process is recurrent when $\mu_0=0$, or if absorption at $-1$ occurs with probability one when $\mu_0>0$), then $q(n)\equiv1$, and consequently $T=B$. The convergence or divergence of the series $T$ is determined by whether the hitting-probability sequence belongs to $\ell_\pi^2(\mathbb{N}_0)$.
\end{remark}

\begin{remark}
In the discrete-time case studied in \cite[Theorem 2.1]{GdI}, the stochastic UL factorization of a transition probability matrix $P$ is possible only when a free parameter, denoted by $y_0$ in \cite{GdI}, satisfies $0<y_0\leq H$, where $H$ is the continued fraction
\begin{equation*}
\mbox{$H=\contFrac*{1 // a_0 // 1 // c_1 // 1 // a_1 // 1 // c_2 // \cdots}$}.
\end{equation*}
Here $(a_n)_{n\in\mathbb{N}_0}$ and $(c_n)_{n\in\mathbb{N}}$ denote the birth and death probabilities, respectively. Also $(b_n)_{n\in\mathbb{N}_0}$ will denote the one-step probabilities of remaining in the same state. Since $P$ is stochastic, we clearly have $a_0+b_0=1$ and $a_n+b_n+c_n=1$ for $n\geq1$. The continued fraction $1/H$ is what is called an \emph{$S$-fraction} (see \cite{JT}). Moreover, its \emph{even contraction} can be written as
\begin{equation}\label{HJF}
\frac{1}{\mbox{$\contFrac*{b_0 // a_0c_1 // b_1 // a_1c_2 // b_2 // a_2c_3 // b_3 // a_3c_4 // \cdots}$}},
\end{equation}
which is a \emph{$J$-fraction}. On the other hand, the quantity $m_{-1}$ in \eqref{mmm1} coincides with the Laplace
transform of $P_{00}(t)$ at the origin, $m_{-1}=\int_0^\infty P_{00}(t)\,dt$
(see \cite[Remark 3.17]{MDIB} or \cite[Chapter 11]{JT}); in particular $m_{-1}=A$ when $\mu_0=0$. In this case $A$ can be represented as the $J$-fraction
\begin{equation*}
A=\frac{1}{\mbox{$\contFrac*{\lambda_0 // \lambda_0\mu_1 // \lambda_1+\mu_1 // \lambda_1\mu_2 // \lambda_2+\mu_2 // \lambda_2\mu_3 // \lambda_3+\mu_3 // \lambda_3\mu_4 // \cdots}$}}.
\end{equation*}
Therefore, the even contraction of the $S$-fraction $1/H$ in \eqref{HJF} coincides with the above $J$-fraction representation of $A$ after replacing $a_n$, $b_n$, and $c_n$ by $\lambda_n$, $\lambda_n+\mu_n$, and $\mu_n$, respectively. This provides a nice connection between the discrete- and continuous-time settings. However, our framework is considerably richer, since in general we allow $\mu_0 \geq 0$.
\end{remark}

As in the LU case, Theorem~\ref{theom37} only guarantees that $\widetilde{\mathcal A}$ is the generator of a birth--death process. In contrast with Proposition~\ref{propSLU}, however, no additional condition is needed in the UL case: the Darboux-transformed process always inherits the standing condition \eqref{uniqueness_assumption}, as the following result shows.

\begin{proposition}\label{propSUL}
Under the assumptions of Theorem~\ref{theom37}, for every admissible choice of $x_0$ and $\tilde\mu_0$ the Darboux transform $\widetilde{\mathcal A}$ satisfies the standing condition \eqref{uniqueness_assumption}. More precisely, with the convention $1/\infty=0$,
\begin{equation}\label{StildeUL}
\widetilde S=\sum_{n=0}^\infty\tilde\pi_n\tilde Q_n^2(0)
=\frac{\mu_0+\frac1A}{\mu_0+\frac1A-(x_0+\tilde\mu_0)}=\frac1M,
\end{equation}
where $M$ is the mass at $x=0$ of $d\widetilde\psi$ in \eqref{spmec2e}.
\end{proposition}

\begin{proof}
Write $X=x_0+\tilde\mu_0$ and \(d_n=X-\tilde\mu_0\sum_{k=0}^n\pi_ku_k^2,n\ge0, d_{-1}=X .\) By \eqref{xxnn2}, \eqref{yynn2} and \eqref{ssnn2},
\[
x_n=\frac{d_n}{\pi_nu_n},\quad y_n=-\frac{d_{n-1}}{\pi_nu_n},\quad n\ge0,\quad
s_n=\frac{\lambda_{n-1}\pi_{n-1}u_{n-1}}{d_{n-1}},\quad n\ge1,
\]
where the case $n=0$ of the formula for $y_n$ is $y_0=-X$. By Theorem~\ref{theom37} we have
$x_n>0$ and $u_n>0$, hence $d_n>0$ for all $n\ge-1$. Now, from
$\tilde\pi_n=y_n\pi_n/(y_0s_n)$ (see the formula after \eqref{mmm1}), we have
\begin{equation}\label{pitildeUL}
\tilde\pi_0=1,\quad
\tilde\pi_n=\frac{d_{n-1}^2}{X\lambda_{n-1}\pi_{n-1}u_{n-1}u_n},\quad n\ge1,
\end{equation}
and, since $\tilde\lambda_n=s_nx_n$, we obtain
\begin{equation}\label{lpitildeUL}
\tilde\lambda_n\tilde\pi_n=\frac{x_ny_n\pi_n}{y_0}=\frac{d_{n-1}d_n}{X\pi_nu_n^2},\quad n\ge0,
\end{equation}
where for $n=0$ this reads $\tilde\lambda_0=x_0=d_0$, in agreement with \eqref{laconcha3r}. Also, evaluating the polynomial identity
\eqref{QDTUL} at $x=0$ gives
\begin{equation}\label{QtildeUL0}
\tilde Q_n(0)=\frac{X}{d_{n-1}},\quad n\ge0 .
\end{equation}
Let $a_N=\sum_{k=0}^{N-1}(\lambda_k\pi_k)^{-1}$, so that
$u_N=1+(\mu_0-X)a_N$ by \eqref{uun}, \eqref{Qn0} and \eqref{Qn00} and $Q_N(0)=1+\mu_0a_N$ by \eqref{Qn0}. From
\eqref{pitildeUL} and \eqref{QtildeUL0},
\[
\sum_{n=0}^{N}\tilde\pi_n\tilde Q_n^2(0)=1+X\sum_{n=0}^{N-1}\frac{1}{\lambda_n\pi_nu_nu_{n+1}} .
\]
We claim that
\begin{equation}\label{telesUL}
\sum_{n=0}^{N-1}\frac{1}{\lambda_n\pi_nu_nu_{n+1}}=\frac{a_N}{u_N},\quad N\ge0 .
\end{equation}
This is trivial for $N=0$. For the induction step, observe that
\[
\frac{a_{N+1}}{u_{N+1}}-\frac{a_N}{u_N}
=\frac{a_{N+1}\big(1+(\mu_0-X)a_N\big)-a_N\big(1+(\mu_0-X)a_{N+1}\big)}{u_Nu_{N+1}}
=\frac{a_{N+1}-a_N}{u_Nu_{N+1}}=\frac{1}{\lambda_N\pi_Nu_Nu_{N+1}} .
\]
Consequently,
\begin{equation}\label{partialStilde}
\sum_{n=0}^{N}\tilde\pi_n\tilde Q_n^2(0)=1+\frac{Xa_N}{u_N}=\frac{1+\mu_0a_N}{1+(\mu_0-X)a_N}=\frac{Q_N(0)}{u_N} .
\end{equation}
Let $N\to\infty$, so that $a_N\uparrow A$. If $A<\infty$, then
\[
\widetilde S=\frac{1+\mu_0A}{1+(\mu_0-X)A}=\frac{\mu_0+\frac1A}{\mu_0+\frac1A-X},
\]
which is infinite if and only if $X=\mu_0+\frac1A$. If $A=\infty$, then $\mu_0>0$ and
$X\le\mu_0$ by \eqref{condULgen}. If $X<\mu_0$, then \eqref{partialStilde} tends to
$\mu_0/(\mu_0-X)$; if $X=\mu_0$, then $u_N=1$ and $Q_N(0)\to\infty$. In all cases
\eqref{StildeUL} holds. Finally, by \eqref{spmec2e} and \eqref{mmm1},
$M=1-X\big(\mu_0+\frac1A\big)^{-1}=1/\widetilde S$.

Therefore, in case (1) of Theorem~\ref{theom37} we have $\tilde\mu_0=0$, hence $d_n=X$ for all $n$, and
\eqref{lpitildeUL} gives
\[
\widetilde A=\sum_{n=0}^\infty\frac{1}{\tilde\lambda_n\tilde\pi_n}=\frac{1}{X}\sum_{n=0}^{\infty}\pi_nu_n^2=\frac{T}{X}=\infty,
\]
so that $\widetilde A+\widetilde B=\infty$. In case (2) we have $X=\mu_0+\frac1A$, so
$\widetilde S=\infty$ by \eqref{StildeUL}. If $\tilde\mu_0>0$ this is exactly
\eqref{uniqueness_assumption}. If $\tilde\mu_0=0$, then $\tilde Q_n(0)=1$ by
\eqref{QtildeUL0}, so $\widetilde B=\widetilde S=\infty$. Hence $\widetilde{\mathcal A}$
satisfies \eqref{uniqueness_assumption} in all cases.
\end{proof}

\begin{remark}
Propositions~\ref{propSLU} and~\ref{propSUL} concern the determinacy of the spectral
measure of the Darboux-transformed process, which is the property needed for its
spectral analysis. A different question is whether the absence of explosion is
preserved by Darboux transformations. This is not the case in general: the Darboux
transform may move the process towards infinity much faster than the original one.
For instance, if $\mu_0=0$, $\lambda_n\pi_n=1$ and $\pi_n=e^{-n^2}$, then the original
process is positive recurrent, while its dual process (see Remark~\ref{remdual}),
which is obtained from the LU factorization with $\hat\mu_0=\lambda_0$, 
explodes, even though it still satisfies \eqref{uniqueness_assumption}
(indeed, by \eqref{ShatLU}, $\widehat S=\lambda_0\sum_n\Sigma_n^2=\infty$, since $\lambda_n\pi_n=1$ and $\Sigma_n\ge1$). Since all our results are formulated for the minimal process, this phenomenon does not affect them, and we do not pursue it further here (see Section~\ref{secConc}). We only note that it cannot occur when the birth rates of the transformed process are bounded or grow at most linearly, as in the $M/M/1$ queue studied in Section~\ref{secMM1}, where all the birth rates are bounded.
\end{remark}

\subsection{Probabilistic interpretation of the coefficients of the UL factorization}\label{picul} Using the results obtained in Section \ref{sec2} we may write the coefficients $x_n$, $y_n$, $s_n$, and $r_n$ in terms of expected values associated with certain stopping times. As we will see, the analysis in this setting requires a more delicate treatment. First, observe that the sequence $(u_n)_{n\in\mathbb{N}_0}$ in \eqref{uun} can be expressed, using \eqref{Qn0} and \eqref{Qn00}, as
\begin{equation}\label{uunext}
u_n = 1 + (\mu_0 - x_0 - \tilde{\mu}_0)\sum_{k=0}^{n-1}\frac{1}{\lambda_k \pi_k}, 
\quad u_0 = 1, \quad n \ge 1.
\end{equation}
Depending on the value of $x_0 + \tilde{\mu}_0$ in condition \eqref{cond1ULp}, different cases may arise:

\begin{enumerate}[leftmargin=0.3in]
\item If $0 < x_0 + \tilde\mu_0 < \mu_0$, then $\mu_0 - x_0 - \tilde\mu_0 > 0$, and we immediately have $u_n > 1$ for all $n \ge 1$. Let us set $x_0 + \tilde\mu_0 = \mu_0 - \alpha$ for some $0 < \alpha < \mu_0$. Then \eqref{uunext} can be written as
\begin{equation*}
u_n = 1 + \alpha\sum_{k=0}^{n-1}\frac{1}{\lambda_k \pi_k}, \quad n \ge 1, \quad u_0 = 1.
\end{equation*}
Observe that this expression coincides with $Q_n(0)$ in \eqref{Qn0}, except that $\mu_0$ is replaced by $\alpha$. Recall the notation $\mathcal A^{(\alpha)}$, $\mathbb P^{(\alpha)}$, $\mathbb E^{(\alpha)}$ introduced in Section~\ref{sec2}. In this setting, from Lemma~\ref{lem21} and Propositions~\ref{prop22} and~\ref{prop23}, we have
\begin{align*}
\mathbb{P}_i^{(\alpha)}\left(\tau_{n+1}<\tau_{-1}\right)&=\frac{u_i}{u_{n+1}}, \quad \mathbb{E}_{n}^{(\alpha)}\left[T_{n}^{(-1)}\right]=\frac{\pi_nu_n}{\alpha},\\
\mathbb{E}_n^{(\alpha)}\left[\tau_{n+1}\mid\tau_{n+1}<\tau_{-1}\right]&=\frac{1}{\lambda_n\pi_nu_nu_{n+1}}\sum_{k=0}^n\pi_ku_k^2.
\end{align*}
With this notation, and using the definition of $x_n$ in~\eqref{xxnn2}, we obtain
\begin{equation*}
x_n=\frac{\mu_0-\alpha}{\alpha\mathbb{E}_{n}^{(\alpha)}\left[T_{n}^{(-1)}\right]}
-\tilde\mu_0\frac{\lambda_n\mathbb{E}_n^{(\alpha)}\left[\tau_{n+1}\mid\tau_{n+1}<\tau_{-1}\right]}
{\mathbb{P}_0^{(\alpha)}\left(\tau_{n+1}<\tau_{-1}\right)}, \quad n\ge0.
\end{equation*}
Now, from \eqref{yynn2} and using the identity 
\begin{equation}\label{idelmxx}
\lambda_nu_{n+1}\mathbb{E}_n^{(\alpha)}\left[\tau_{n+1}\mid\tau_{n+1}<\tau_{-1}\right]
-\mu_nu_{n-1}\mathbb{E}_{n-1}^{(\alpha)}\left[\tau_{n}\mid\tau_{n}<\tau_{-1}\right]=u_n,\quad n\geq1,
\end{equation}
we obtain \(y_0=-(x_0+\tilde\mu_0)\) and
\begin{equation*}
y_n=-\frac{\mu_0-\alpha}{\alpha\mathbb{E}_{n}^{(\alpha)}\left[T_{n}^{(-1)}\right]}
+\tilde\mu_0\frac{\mu_n\mathbb{E}_{n-1}^{(\alpha)}\left[\tau_{n}\mid\tau_{n}<\tau_{-1}\right]}
{\mathbb{P}_0^{(\alpha)}\left(\tau_{n-1}<\tau_{-1}\right)}, \quad n\ge1.
\end{equation*}
Finally, from \eqref{ssnn2} and \eqref{rrnn2}, we have
\begin{equation*}
s_n=\left[\frac{\mu_0-\alpha}{\alpha\lambda_{n-1}\mathbb{E}_{n-1}^{(\alpha)}\left[T_{n-1}^{(-1)}\right]}
-\tilde\mu_0\frac{\mathbb{E}_{n-1}^{(\alpha)}\left[\tau_{n}\mid\tau_{n}<\tau_{-1}\right]}
{\mathbb{P}_0^{(\alpha)}\left(\tau_{n}<\tau_{-1}\right)}\right]^{-1}, \quad n\ge1, \quad s_0=1,
\end{equation*}
and
\begin{equation*}
r_n=\left[-\frac{\mu_0-\alpha}{\alpha\mu_{n}\mathbb{E}_{n}^{(\alpha)}\left[T_{n}^{(-1)}\right]}
+\tilde\mu_0\frac{\mathbb{E}_{n-1}^{(\alpha)}\left[\tau_{n}\mid\tau_{n}<\tau_{-1}\right]}
{\mathbb{P}_0^{(\alpha)}\left(\tau_{n-1}<\tau_{-1}\right)}\right]^{-1}, \quad n\ge1.
\end{equation*}
\smallskip

\item If $0 < x_0 + \tilde\mu_0 = \mu_0$, then we immediately have $u_n \equiv 1$. This corresponds to the limiting case of the previous situation when $\alpha \to 0$, where $\mathcal{A}^{(0)}$ becomes the infinitesimal generator of a conservative birth--death process. In this case, absorption cannot occur. Using formulas \eqref{Expx} and \eqref{idelmxx} (for $\alpha=0$ and $u_n=1$), and applying \eqref{xxnn2} and \eqref{yynn2}--\eqref{rrnn2}, we obtain
\begin{equation*}
x_n=\frac{\mu_0}{\pi_n}-\tilde\mu_0\lambda_n\mathbb{E}_n^{(0)}[\tau_{n+1}], \quad n\geq0,\quad y_0=-\mu_0,\quad
y_n=-\frac{\mu_0}{\pi_n}+\tilde\mu_0\mu_n\mathbb{E}_{n-1}^{(0)}[\tau_n], \quad n\ge1,
\end{equation*}
and
\begin{equation*}
s_0=1, \quad 
s_n=-r_n=\left[\frac{\mu_0}{\mu_{n}\pi_{n}}-\tilde\mu_0\mathbb{E}_{n-1}^{(0)}[\tau_{n}]\right]^{-1}, \quad n\ge1.
\end{equation*}
Observe also that in this situation $x_0=\mu_0-\tilde\mu_0$ is determined by $\tilde\mu_0$ (and $\tilde\mu_0=0$ unless $A=\infty$ and $B<\infty$, in which case $0\le\tilde\mu_0\le\mu_0/B$).
\smallskip

\item Let us now consider the case $\mu_0 < x_0 + \tilde\mu_0 < \mu_0 + \frac{1}{A}$. Here we must assume \(A<\infty\); equivalently, the process is transient when \(\mu_0=0\), while absorption at \(-1\) is not certain when \(\mu_0>0\). If \(A=\infty\), this case cannot occur. In this situation, we have $\mu_0 - x_0 - \tilde\mu_0 < 0$, which immediately implies that $u_n < 1$ for all $n \ge 1$ (see \eqref{uunext}). Moreover, note that the limiting case $x_0 + \tilde\mu_0 \to \mu_0 + \frac{1}{A}$ yields the probability $\mathbb{P}_n(\tau_0 < \infty)$ given in \eqref{qqprob}. Therefore,
$$
\mathbb{P}_n(\tau_0 < \infty) < u_n < 1,\quad n\geq1.
$$
This suggests a probabilistic interpretation of \(u_n\), which can in fact be made exact as follows. Let us set $x_0 + \tilde\mu_0 = \mu_0 + \omega$ for some $0 < \omega < \frac{1}{A}$. Then $u_n$ can be written as 
\begin{equation*}
u_n = 1 -\omega\sum_{k=0}^{n-1}\frac{1}{\lambda_k \pi_k}, \quad n \ge 1, \quad u_0 = 1.
\end{equation*}
This expression can be written as 
\[
u_n
=
1-\omega A
+
\omega\sum_{k=n}^{\infty}\frac{1}{\lambda_k\pi_k},
\]
where $A$ is defined by \eqref{quantity_AB}. Using \eqref{qqprob} we obtain
\begin{equation}\label{uunprob}
u_n
=
\mathbb{P}_n(\tau_0<\infty)
+
(1-\omega A)\mathbb{P}_n(\tau_0=\infty),
\quad n\geq0.
\end{equation}
In case (1) (when \(0<x_0+\tilde\mu_0<\mu_0,\))
we were in a situation where the absorption rate to state $-1$ of the original birth--death process had to be modified. When
$x_0+\tilde\mu_0\to0$ (or equivalently $\alpha\to\mu_0$), the absorption rate to state $-1$ equals $\mu_0$ (the original one), whereas when
$x_0+\tilde\mu_0\to\mu_0$ (or $\alpha\to0$), the birth--death process becomes conservative, with no absorption at state $-1$. In the present case, however, the probabilistic interpretation is related to the opposite boundary point, namely $+\infty$. Note that the sequence $u_n$ is harmonic on $\mathbb{N}$ with $u_0=1$ and
\[
u_\infty
:=
\lim_{n\to\infty}u_n
=
1-\omega A.
\]
Thus $u$ may be viewed as the harmonic extension to the interior of the boundary values
\(u_0=1,
u_{\infty}=1-\omega A.
\) In particular, we obtain \eqref{uunprob}. This identity gives a genuine probabilistic interpretation of $u_n$.
Let $\xi_\omega$ be an auxiliary Bernoulli random variable independent of the birth--death process, such that
\[
\mathbb{P}(\xi_\omega=1)=1-\omega A,
\quad
\mathbb{P}(\xi_\omega=0)=\omega A,
\]
and define the event
\begin{equation}\label{eventEomega}
\mathcal{E}_\omega
=
\{\tau_0<\infty\}
\cup
\bigl(
\{\tau_0=\infty\}\cap\{\xi_\omega=1\}
\bigr).
\end{equation}
Then
\begin{equation*}
u_n
=
\mathbb{P}_n(\mathcal{E}_\omega),
\quad n\geq0.
\end{equation*}
Thus $\mathcal{E}_\omega$ may be interpreted as a randomized boundary
event: if the process reaches $0$, the event occurs with probability one,
whereas if the process never reaches $0$, the trajectory is accepted
independently with probability $1-\omega A$. In the limiting case
$\omega\to1/A$, this latter probability tends to zero and
$\mathcal{E}_\omega$ reduces to the event $\{\tau_0<\infty\}$. Conversely,
when $\omega\to0$, every trajectory is accepted and
$\mathbb{P}_n(\mathcal{E}_\omega)\to1$. For the passage-time interpretation, let us consider the conservative birth--death process on $\mathbb{N}_0$ obtained by taking the $u$-transformed rates in the interior and completing the dynamics conservatively at $0$. More precisely, its birth and death rates are
\begin{equation*}
\lambda_n^{u}
=
\lambda_n\frac{u_{n+1}}{u_n},
\quad n\geq0,
\qquad
\mu_n^{u}
=
\mu_n\frac{u_{n-1}}{u_n},
\quad n\geq1,
\qquad
\mu_0^{u}=0,
\end{equation*}
and its potential coefficients are
\(
\pi_n^{u}=\pi_nu_n^2,\ n\geq0,
\)
since $u_0=1$. Therefore, the standard expression for the expected passage time from $n$ to $n+1$ (see \eqref{Expx}) gives
\begin{equation}\label{endcodu}
\mathbb{E}_n^{u}[\tau_{n+1}]
=
\frac{1}{\lambda_n\pi_nu_nu_{n+1}}
\sum_{k=0}^n\pi_ku_k^2,
\quad n\geq0.
\end{equation}
Here $\mathbb{E}_n^{u}$ denotes expectation with respect to this conservative completion of the interior $u$-transformed dynamics. Although
$u_n=\mathbb{P}_n(\mathcal{E}_\omega)$, the expectation in
\eqref{endcodu} should not, in general, be identified with
$\mathbb{E}_n[\tau_{n+1}\mid\mathcal{E}_\omega]$ for the original process. Indeed, $u$ satisfies the harmonic relation only in the interior, and therefore does not define a global Doob transform of the original process under the boundary convention at $0$. 
With this notation, and using the definition of $x_n$ in \eqref{xxnn2},
we obtain
\begin{equation*}
x_n
=
\frac{\mu_0+\omega}
{\pi_n\mathbb{P}_n(\mathcal{E}_\omega)}
-
\tilde\mu_0\lambda_n
\mathbb{P}_{n+1}(\mathcal{E}_\omega)
\mathbb{E}_n^{u}[\tau_{n+1}],
\quad n\geq0.
\end{equation*}
Now, from \eqref{yynn2} and using the identity \eqref{idelmxx} adapted to this situation, we obtain
\begin{equation*}
y_n
=
-\frac{\mu_0+\omega}
{\pi_n\mathbb{P}_n(\mathcal{E}_\omega)}
+
\tilde\mu_0\mu_n
\mathbb{P}_{n-1}(\mathcal{E}_\omega)
\mathbb{E}_{n-1}^{u}[\tau_n],
\quad n\geq1, \quad y_0=-(\mu_0+\omega).
\end{equation*}
Finally, from \eqref{ssnn2} and \eqref{rrnn2}, we have
\begin{equation*}
s_n
=
\left[
\frac{\mu_0+\omega}
{\mu_n\pi_n\mathbb{P}_{n-1}(\mathcal{E}_\omega)}
-
\tilde\mu_0
\mathbb{P}_{n}(\mathcal{E}_\omega)
\mathbb{E}_{n-1}^{u}[\tau_n]
\right]^{-1},
\quad n\geq1,
\quad
s_0=1,
\end{equation*}
and
\begin{equation*}
r_n
=
\left[
-\frac{\mu_0+\omega}
{\mu_n\pi_n\mathbb{P}_{n}(\mathcal{E}_\omega)}
+
\tilde\mu_0
\mathbb{P}_{n-1}(\mathcal{E}_\omega)
\mathbb{E}_{n-1}^{u}[\tau_n]
\right]^{-1},
\quad n\geq1.
\end{equation*}
\smallskip

\item As in the previous case, if
\(
x_0+\tilde\mu_0
=
\mu_0+\frac{1}{A},
\)
then
\(
u_n
=
q(n)
=
\mathbb{P}_n(\tau_0<\infty)
\)
for all $n\geq0$. This corresponds to the limiting case $\omega\to1/A$ of the previous situation, in which the randomized event $\mathcal{E}_\omega$ defined in \eqref{eventEomega} reduces to
\(
\{\tau_0<\infty\}.
\)
More precisely, $q$ satisfies the Dirichlet problem
\(
\mathcal{A}q(n)=0,\ n\geq1,
\)
with boundary condition $q(0)=1$, or equivalently, $q$ is harmonic for the process \emph{stopped upon first reaching $0$}; see Lemma~\ref{lem23}. Since $A<\infty$ in this situation, $q$ is given explicitly by \eqref{qqprob}. For the passage-time interpretation, let us consider the conservative birth--death process on $\mathbb{N}_0$ obtained by taking the $q$-transformed rates in the interior and completing the dynamics conservatively at $0$. Its birth and death rates are
\[
\lambda_n^{q}
=
\lambda_n\frac{q(n+1)}{q(n)},
\quad n\geq0,
\qquad
\mu_n^{q}
=
\mu_n\frac{q(n-1)}{q(n)},
\quad n\geq1,
\qquad
\mu_0^{q}=0,
\]
and its potential coefficients are
\(
\pi_n^{q}
=
\pi_nq^2(n),\ n\geq0.
\)
Hence,
\begin{equation}\label{endcod2}
\mathbb{E}_n^{q}[\tau_{n+1}]
=
\frac{1}
{\lambda_n\pi_nq(n)q(n+1)}
\sum_{k=0}^n\pi_kq^2(k),
\quad n\geq0.
\end{equation}
Here $\mathbb{E}_n^{q}$ denotes expectation with respect to this conservative completion of the interior $q$-transformed dynamics. Although
$q(n)=\mathbb{P}_n(\tau_0<\infty)$, the expectation in \eqref{endcod2} should not, in general, be identified with
\(
\mathbb{E}_n[\tau_{n+1}\mid\tau_0<\infty]
\)
for the original process, since $q$ satisfies the harmonic relation only on the interior $n\geq1$ and therefore does not define a global Doob transform under the boundary convention at $0$. Using \eqref{qqprob}, \eqref{endcod2}, and applying
\eqref{xxnn2} and \eqref{yynn2}--\eqref{rrnn2}, we obtain
\begin{equation*}
x_n
=
\frac{\mu_0+\frac{1}{A}}
{\pi_n\mathbb{P}_n(\tau_0<\infty)}
-
\tilde\mu_0\lambda_n
\mathbb{P}_{n+1}(\tau_0<\infty)
\mathbb{E}_n^{q}[\tau_{n+1}],
\quad n\geq0,
\end{equation*}
\begin{equation*}
y_n
=
-\frac{\mu_0+\frac{1}{A}}
{\pi_n\mathbb{P}_n(\tau_0<\infty)}
+
\tilde\mu_0\mu_n
\mathbb{P}_{n-1}(\tau_0<\infty)
\mathbb{E}_{n-1}^{q}[\tau_n],
\quad n\geq1,\quad y_0=-\left(\mu_0+\frac{1}{A}\right).
\end{equation*}
Moreover,
\begin{equation*}
s_n
=
\left[
\frac{\mu_0+\frac{1}{A}}
{\mu_n\pi_n\mathbb{P}_{n-1}(\tau_0<\infty)}
-
\tilde\mu_0
\mathbb{P}_{n}(\tau_0<\infty)
\mathbb{E}_{n-1}^{q}[\tau_n]
\right]^{-1},
\quad n\geq1,
\quad
s_0=1,
\end{equation*}
and
\begin{equation*}
r_n
=
\left[
-\frac{\mu_0+\frac{1}{A}}
{\mu_n\pi_n\mathbb{P}_{n}(\tau_0<\infty)}
+
\tilde\mu_0
\mathbb{P}_{n-1}(\tau_0<\infty)
\mathbb{E}_{n-1}^{q}[\tau_n]
\right]^{-1},
\quad n\geq1.
\end{equation*}
Observe also that in this situation $x_0=\mu_0+\frac{1}{A}-\tilde\mu_0$ is determined
by $\tilde\mu_0$ (and $\tilde\mu_0=0$ unless $T<\infty$, where $T$ is given by the first
expression in \eqref{TTT3}, in which case
$0\le\tilde\mu_0\le\frac{1}{T}\big(\mu_0+\frac{1}{A}\big)$).

\end{enumerate}

\medskip

As in the case of the LU factorization, we can obtain further
simplifications for the coefficients $x_n$, $y_n$, $s_n$, and $r_n$ by
assuming $\mu_0=0$ and/or $\tilde\mu_0=0$, simply by substituting
these values into the previous formulas. The most significant
simplification, apart from the one described for the dual process in
Remark~\ref{remdual2}, arises when
\(
\mu_0=\tilde\mu_0=0,
\)
(which is possible if and only if $A<\infty$). In this
case we have $t_n=0$ and $q_n=-u_n/u_{n-1}$. Therefore, choosing
$\omega=x_0$ within the range
\(
0<\omega\leq\frac{1}{A}
\) (cases (3) and (4)), we obtain
\(
u_n=\mathbb{P}_n(\mathcal{E}_\omega),
\)
and hence
\[
x_n=-y_n
=
\frac{\omega}
{\pi_n\,\mathbb{P}_n(\mathcal{E}_\omega)},
\; n\geq0,\quad
s_n
=
\frac{\mu_n\pi_n}{\omega}
\,\mathbb{P}_{n-1}(\mathcal{E}_\omega),
\; n\geq1,
\;
s_0=1,\quad
r_n
=
-\frac{\mu_n\pi_n}{\omega}
\,\mathbb{P}_{n}(\mathcal{E}_\omega),
\quad n\geq1.
\]
The birth--death rates for the Darboux-transformed process are given in this situation (see \eqref{newbdcoe2} and \eqref{laconcha3r}) by
\(
\tilde\lambda_0=\omega,
\)
and
\[
\tilde\lambda_n
=
\mu_n
\frac{\mathbb{P}_{n-1}(\mathcal{E}_\omega)}
{\mathbb{P}_{n}(\mathcal{E}_\omega)},
\; n\geq1,\quad
\tilde\mu_{n+1}
=
\lambda_n
\frac{\mathbb{P}_{n+1}(\mathcal{E}_\omega)}
{\mathbb{P}_{n}(\mathcal{E}_\omega)},
\; n\geq0.
\]
If
\(
\omega=\frac{1}{A},
\)
then
\(
\mathcal{E}_{1/A}=\{\tau_0<\infty\},
\)
and all the previous formulas remain valid after replacing
$\mathbb{P}_n(\mathcal{E}_\omega)$ by
$\mathbb{P}_n(\tau_0<\infty)$.

\begin{remark}
The probabilistic interpretation of the coefficients arising in the UL factorization is another new feature with respect to the discrete-time case. Although in that setting the coefficients were themselves transition probabilities, no connections with hitting probabilities or expected stopping times were established in \cite{GdI}. 
\end{remark}

In the following section we illustrate the general results of Sections~\ref{Sec3} and~\ref{Sec4} with a classical example, the $M/M/1$ queue. Since its spectral measure and orthogonal polynomials are explicitly known (see \cite[Section~3.6.2]{MDIB}), all the quantities involved can be computed in closed form: the coefficients of the factors, the admissible ranges of the parameters, and the birth--death rates, spectral measures and orthogonal polynomials of the Darboux-transformed processes.

\section{The $M/M/1$ queue}\label{secMM1}

Let  $\{X_t, t\geq0\}$ be the possibly nonconservative birth--death process with constant birth--death rates 
$$
\lambda_n=\lambda,\quad n\geq0,\quad\mu_n=\mu,\quad n\geq1,\quad\lambda,\mu>0,\quad\mu_0\geq0.
$$
The infinitesimal generator \eqref{QQmm} is given by
\begin{equation}\label{A_Ex}
	\mathcal{A}=\begin{pmatrix}
		-(\lambda+\mu_0) & \lambda &  &  & \\
		\mu & -(\lambda+\mu) & \lambda &  &\\
		0 & \mu & -(\lambda+\mu) & \lambda &\\
	& & \ddots & \ddots & \ddots
	\end{pmatrix}.
\end{equation}
When \(\mu_0>0\), the process may jump from state \(0\) to the absorbing state \(-1\); conditional on leaving \(0\), this occurs with probability \(\mu_0/(\lambda+\mu_0)\). If we denote by $d\psi$ the spectral measure associated with $\mathcal{A}$, then we have 
\begin{equation*}
    m(z;\psi)=\frac{1}{\lambda+\mu_0-z-\lambda\mu m(z;\psi^{(0)})},
\end{equation*}
where $m(z;\psi):=\int_0^\infty(x-z)^{-1}d\psi(x), z\in\mathbb{C}\setminus[0,\infty)$ is the \emph{Stieltjes transform} of the spectral measure $d\psi$ and \( m(z;\psi^{(0)})\) the Stieltjes transform of the \emph{0-associated operator}, obtained by deleting the first row and column of \(\mathcal A\). Using the explicit expression of \( m(z;\psi^{(0)})\) (which can be found in \cite[(3.98)]{MDIB}), we get, after some straightforward computations
\begin{equation}\label{Bzpsiex}
m(z;\psi)=\frac{\lambda - \mu - z + 2\mu_0 - \sqrt{\lambda^{2} - 2\lambda\mu - 2\lambda z + \mu^{2} - 2\mu z + z^{2}}}
{2\left[(\mu-\mu_0) z+\mu_0(\lambda  - \mu  + \mu_0)\right]}.
\end{equation}
For $\mu_0\neq\mu$, the only possible pole of $m(z;\psi)$ is located at
\begin{equation}\label{zzz}
\zeta=\mu_0\left(1-\frac{\lambda}{\mu-\mu_0}\right).
\end{equation}
A discrete mass at $\zeta$ occurs only when $\lambda\mu<(\mu-\mu_0)^2$. Therefore, for $\mu_0\neq\mu$, by the Perron--Stieltjes inversion formula (see \cite[Proposition~1.1]{MDIB}), the spectral measure is given by
\begin{equation}\label{SMEEx1}
d\psi(x)
=
\frac{\sqrt{4\lambda\mu-(\lambda+\mu-x)^2}}
{2\pi[(\mu-\mu_0)x+\mu_0(\lambda-\mu+\mu_0)]}
\mathbf{1}_{\{x\in[\sigma_-,\sigma_+]\}}\,dx+
\left(1-\frac{\lambda\mu}{(\mu-\mu_0)^2}\right)
\mathbf{1}_{\{\lambda\mu<(\mu-\mu_0)^2\}}
\,\delta_\zeta(dx),
\end{equation}
where $\sigma_{\pm}=(\sqrt{\lambda}\pm\sqrt{\mu})^2$, $\mathbf{1}_A$ denotes the indicator function, and \(\delta_a\) denotes the Dirac measure at \(a\). In particular, when $\mu_0=\mu$, there is no discrete component and \eqref{SMEEx1} reduces to its absolutely continuous part. It is also known (see~\cite[(35)]{dlI2022}) that the orthogonal polynomials associated with this spectral measure are given by
\begin{equation*}
Q_n(x)=\left(\frac{\mu}{\lambda}\right)^{n/2}\left[U_n\left(\frac{\lambda+\mu-x}{2\sqrt{\lambda\mu}}\right)+\frac{\mu_0-\mu}{\lambda}\sqrt{\frac{\lambda}{\mu}}\,U_{n-1}\left(\frac{\lambda+\mu-x}{2\sqrt{\lambda\mu}}\right)\right],\quad n\geq 0,
\end{equation*}
where $(U_n)_{n\in\mathbb{N}_0}$ are the \emph{Chebyshev polynomials of the second kind}, with the standard convention that $U_{-1}(x)=0$. Observe that when $\mu_0=\mu$ the second term in the previous expression vanishes, so that
the polynomials $(Q_n)_{n\in\mathbb{N}_0}$ simplify considerably.

The potential coefficients \eqref{potcoef} are now given by
\begin{equation}\label{potcoefx1}
\pi_n=\left(\frac{\lambda}{\mu}\right)^n,\quad n\geq0.
\end{equation}
For $\mu_0\geq0$, the value of $Q_n(0)$ in \eqref{Qn0} can be computed explicitly, and is given by
\begin{equation*}
Q_n(0)=\begin{cases}
1+\dfrac{\mu_0}{\lambda-\mu}\left[1-\left(\dfrac{\mu}{\lambda}\right)^n\right],&\mbox{if}\quad \lambda\neq\mu\\[1em]
1+\dfrac{\mu_0}{\lambda}n,&\mbox{if}\quad \lambda=\mu\\
\end{cases},\quad n\geq0,
\end{equation*}
while the value of $Q_n^{(0)}(0)$ in \eqref{Qn00} is given by
\begin{equation*}
Q_n^{(0)}(0)=\begin{cases}
\dfrac{1}{\mu-\lambda}\left[1-\left(\dfrac{\mu}{\lambda}\right)^n\right],&\mbox{if}\quad \lambda\neq\mu\\[1em]
-\dfrac{n}{\lambda},&\mbox{if}\quad \lambda=\mu\\
\end{cases},\quad n\geq0.
\end{equation*}
The quantities $A$ and $B$ in \eqref{quantity_AB} are given, using \eqref{potcoefx1}, by
\begin{equation}\label{quantABex}
A=\begin{cases}
\dfrac{1}{\lambda-\mu},&\mbox{if}\quad \mu<\lambda\\[1em]
\infty,&\mbox{if}\quad \mu\geq\lambda
\end{cases},\qquad B=\begin{cases}
\infty,&\mbox{if}\quad \mu\leq\lambda\\[1em]
\dfrac{\mu}{\mu-\lambda},&\mbox{if}\quad \mu>\lambda
\end{cases}.
\end{equation}
From these quantities, when $\mu_0 = 0$, we deduce that the process is recurrent if and only if $\mu \ge \lambda$. In this case, the process is positive recurrent whenever $\mu > \lambda$, and null recurrent when $\lambda = \mu$. When $\mu_0 > 0$, absorption at $-1$ is certain if and only if $\mu \ge \lambda$, in which case the absorption is ergodic whenever $\mu > \lambda$. In this setting, we can also compute, using the previous expressions, the following sum:
\begin{equation}\label{pqm0ex}
\sum_{k=0}^n \pi_k Q_k^2(0) =
\begin{cases}
\begin{split}
&\dfrac{\mu \mu_0^2 \left( \dfrac{\mu}{\lambda} \right)^n - \lambda (\mu - \lambda - \mu_0)^2 \left( \dfrac{\lambda}{\mu} \right)^n}{(\mu - \lambda)^3}\\[0.5em]
&\qquad - \dfrac{\mu_0(\lambda-\mu+\mu_0)(2n+1)-(\mu-\lambda)(\mu-\mu_0)+2\lambda\mu_0}{(\mu - \lambda)^2}
\end{split},
& \text{if } \;\lambda \ne \mu, \\[3em]
\dfrac{(n+1)(2\mu_0^2n^2+\mu_0(\mu_0+6\lambda)n+6\lambda^2)}{6\lambda^2}, & \text{if } \; \lambda = \mu.
\end{cases}
\end{equation}
It is possible to see that, for $\mu_0>0$, $S=\sum_{k=0}^{\infty}\pi_kQ_k^2(0)=\infty$ for all $\lambda,\mu>0$. Hence, the associated moment problem is determinate. On the other hand, when $\mu_0=0$, it follows from \eqref{quantABex} that $A+B=\infty$, and therefore the associated Stieltjes moment problem is determinate. Moreover, since the birth rates are constant, the process is not explosive, and the same holds for all its Darboux transforms considered below, whose birth rates are bounded (see Remark~\ref{remexplosion}).

\medskip

\hspace{-.35cm}\underline{\textbf{LU factorization}}. The coefficients $\tilde s_n, \tilde r_n, \tilde x_n,$ and $\tilde y_n$ in \eqref{ssnn} and \eqref{rrnn}-\eqref{yynn} can be explicitly computed from \eqref{potcoefx1}--\eqref{pqm0ex}. The resulting expressions are generally cumbersome, except in certain particular cases. To ensure that we obtain a Darboux-transformed birth--death process with generator $\widehat{\mathcal{A}}$ (see Theorem \ref{theom33}), two situations may arise:
\begin{itemize}[leftmargin=0.25in]
\item If $\mu_0 > 0$ (nonconservative case), the parameter $\hat\mu_0$ must satisfy \eqref{condLU0}. Therefore,
  \[
  0 \le \hat\mu_0 \le \lambda + \mu_0.
  \]
  \item If $\mu_0 = 0$ (conservative case), then $Q_n(0) = 1$ for all $n \ge 0$. 
Taking into account the value of $B$ in \eqref{quantABex}, we have two cases (see \eqref{condLU2}):
  \begin{itemize}[leftmargin=0.25in]
  \item If $\mu \le \lambda$, then $B = \infty$, and the parameter $\hat\mu_0$ must satisfy $0 \le \hat\mu_0 \le \lambda$.
  \item If $\mu > \lambda$, then $B = \frac{\mu}{\mu - \lambda}$, and the parameter $\hat\mu_0$ must satisfy $0 \le \hat\mu_0 \le \mu$.
  \end{itemize}
  Combining both situations, we require
  \[
  0 \le \hat\mu_0 \le \max\{\lambda, \mu\}.
  \]
\end{itemize}

Under either of these conditions, the Darboux-transformed operator $\widehat{\mathcal{A}}$ is again the infinitesimal generator of a birth--death process, with birth and death rates $\{\hat{\lambda}_n,\hat{\mu}_n\}$ given in \eqref{newbdcoe}. This process again satisfies \eqref{uniqueness_assumption}. Indeed, by Proposition~\ref{propSLU} it suffices to check that $\widehat S=\infty$, and since $\Sigma_n\ge1$ and $\Sigma_n\ge\pi_nQ_n^2(0)$, we have
\[
\widehat S\ge\frac{\lambda+\mu_0}{\lambda}\sum_{n=0}^\infty\frac{Q_n(0)}{Q_{n+1}(0)}=\infty,
\]
because the ratio $Q_n(0)/Q_{n+1}(0)$ converges to a positive limit (equal to $1$ if $\mu_0=0$ or $\lambda\ge\mu$, and to $\lambda/\mu$ if $\mu_0>0$ and $\lambda<\mu$). The spectral measure for $\widehat{\mathcal{A}}$ is given by \eqref{spmec1} with $\lambda_0+\mu_0=\lambda+\mu_0$. For $\mu_0\neq\mu$, the only possible discrete mass is located at $\zeta$ in \eqref{zzz} and occurs only when $\lambda\mu<(\mu-\mu_0)^2$. Hence,
\begin{equation}\label{SMEEx1DTLU}
\begin{split}
d\widehat{\psi}(x)&=\frac{x\sqrt{4\lambda\mu-(\lambda+\mu-x)^2}}{2\pi(\lambda+\mu_0)[(\mu-\mu_0)x+\mu_0(\lambda-\mu+\mu_0)]}\mathbf{1}_{\{x\in[\sigma_-,\sigma_+]\}}dx\\
&\hspace{1cm}+\frac{\mu_0(\mu-\lambda-\mu_0)((\mu-\mu_0)^2-\lambda\mu)}{(\mu-\mu_0)^3(\lambda+\mu_0)}\mathbf{1}_{\{\lambda\mu<(\mu-\mu_0)^2\}}\delta_\zeta(dx).
\end{split}
\end{equation}
In particular, when $\mu_0=\mu$, there is no discrete component and \eqref{SMEEx1DTLU} reduces to its absolutely continuous part.
The corresponding orthogonal polynomials can be derived from \eqref{Qnhat}, although their explicit form is generally difficult to obtain. We now examine some particular cases where further simplifications arise.

\begin{itemize}[leftmargin=0.25in]
\item Case $\mu_0 = \hat\mu_0 = 0$.  In this case, we have $Q_n(0)=1$. Assume first that $\lambda\neq\mu$. Then, from \eqref{ssnn} and \eqref{rrnn}--\eqref{yynn}, it follows that
\[
\tilde s_n=\frac{1}{\lambda^{n}}\left(\frac{\lambda^{n+1}-\mu^{n+1}}{\lambda-\mu}\right),
\quad
\tilde r_n=-\frac{\mu}{\lambda^{n}}\left(\frac{\lambda^{n}-\mu^{n}}{\lambda-\mu}\right),
\quad
\tilde x_n=-\tilde y_n=\frac{(\lambda-\mu)\lambda^{n+1}}{\lambda^{n+1}-\mu^{n+1}},
\]
where $\tilde s_n,\tilde x_n,\tilde y_n$ are defined for $n\geq0$, whereas $\tilde r_n$ is defined for $n\geq1$. The birth--death rates of the Darboux-transformed process in~\eqref{newbdcoe} are therefore
\begin{equation}\label{yawey}
\hat{\lambda}_n=
\frac{\lambda^{n+2}-\mu^{n+2}}
{\lambda^{n+1}-\mu^{n+1}},
\quad
\hat{\mu}_{n+1}
=
\lambda\mu
\frac{\lambda^{n+1}-\mu^{n+1}}
{\lambda^{n+2}-\mu^{n+2}},
\quad n\geq0,
\end{equation}
with $\hat{\lambda}_n+\hat{\mu}_n=\lambda+\mu$. When $\lambda=\mu$, the previous expressions are understood by taking the limit $\mu\to\lambda$. In this case,
\[
\tilde s_n=n+1,\quad
\tilde r_n=-n,\quad
\tilde x_n=-\tilde y_n=\frac{\lambda}{n+1},
\]
with $\tilde s_n,\tilde x_n,\tilde y_n$ defined for $n\geq0$ and $\tilde r_n$ for $n\geq1$, and the transformed birth--death rates become
\[
\hat{\lambda}_n=\lambda\frac{n+2}{n+1},
\quad
\hat{\mu}_{n+1}=\lambda\frac{n+1}{n+2},
\quad n\geq0.
\]
Again, $\hat{\lambda}_n+\hat{\mu}_n=2\lambda$. The corresponding spectral measure in \eqref{SMEEx1DTLU} is
\begin{equation}\label{psiabs}
d\widehat{\psi}(x)=
\frac{\sqrt{4\lambda\mu-(\lambda+\mu-x)^2}}
{2\pi\lambda\mu}
\mathbf{1}_{\{x\in[\sigma_-,\sigma_+]\}}dx,
\end{equation}
which coincides with the spectral measure of the original birth--death process \eqref{SMEEx1} for $\mu_0=\mu$. For $\lambda\neq\mu$, the associated orthogonal polynomials $(\hat Q_n)_{n\in\mathbb{N}_0}$ are given by
\[
\hat Q_n(x)=
\frac{\lambda^n(\lambda-\mu)}
{\lambda^{n+1}-\mu^{n+1}}
\left(\frac{\mu}{\lambda}\right)^{n/2}
U_n\!\left(
\frac{\lambda+\mu-x}{2\sqrt{\lambda\mu}}
\right),
\quad n\geq0.
\]
For $\lambda=\mu$, this expression reduces to
\[
\hat Q_n(x)=
\frac{1}{n+1}
U_n\!\left(1-\frac{x}{2\lambda}\right),
\quad n\geq0.
\]

In the special case $\lambda<\mu$, the birth--death process is positive recurrent. Let $G_\pi$ be a random variable with the stationary distribution
\[
\mathbb{P}(G_\pi=n)
=
\frac{\mu-\lambda}{\mu}
\left(\frac{\lambda}{\mu}\right)^n,
\quad n\geq0.
\]
Thus, $G_\pi\sim\operatorname{Geo}(1-\lambda/\mu)$, with parameter $p=1-\lambda/\mu$. Following \eqref{ssnnB}, the coefficient $\tilde s_n$ can be written as
\[
\tilde s_n
=
\frac{F_{G_\pi}(n)}
{\mathbb{P}(G_\pi=n)},
\quad n\geq0,
\]
where $F_{G_\pi}(n)=\mathbb{P}(G_\pi\leq n)$ denotes the cumulative distribution function of $G_\pi$. The birth--death rates of the Darboux-transformed process in~\eqref{yawey} can then be written as
\[
\hat{\lambda}_n
=
\mu\frac{F_{G_\pi}(n+1)}{F_{G_\pi}(n)},
\quad
\hat{\mu}_{n+1}
=
\lambda\frac{F_{G_\pi}(n)}{F_{G_\pi}(n+1)},
\quad n\geq0.
\]
As mentioned after \eqref{ssnnB}, the operator $\widehat{\mathcal{A}}$ corresponding to $\hat{\mu}_0=0$, that is, $\widehat{\mathcal{A}}^{(0)}$, is precisely the \emph{Doob $h$-transform} (see~\eqref{doobt}) of the dual process $\mathcal{A}^d$ (or $\widehat{\mathcal{A}}^{(\lambda)}$), obtained by taking $h(n)=F_{G_\pi}(n)$.

\end{itemize}

In the remaining cases where further simplifications occur (namely, $\mu_0 > 0, \hat\mu_0 = 0$ and $\mu_0 > 0, \hat\mu_0 = \lambda + \mu_0$; see the end of Section~\ref{piclu}), explicit expressions for $\hat\lambda_n$ and $\hat\mu_n$ can still be obtained easily. However, in these cases, no additional simplifications arise for the spectral measure of the Darboux transformation or for the associated orthogonal polynomials. Nonetheless, the following situation exhibits an interesting phenomenon.

\begin{itemize}[leftmargin=0.25in]
\item For $\mu > \lambda$, let us consider the case $\mu_0 = \mu - \lambda > 0$. In this setting, the original birth--death process exhibits ergodic absorption at the state $-1$ (see \eqref{quantABex}). The corresponding spectral measure is
\begin{equation*}
d\psi(x) = \frac{\sqrt{4\lambda\mu - (\lambda + \mu - x)^2}}{2\pi\lambda x}\,\mathbf{1}_{\{x\in[\sigma_-,\sigma_+]\}}dx.
\end{equation*}
From \eqref{ssnn} and \eqref{rrnn}--\eqref{yynn}, we obtain
$$
\tilde s_n = -\tilde r_{n+1} = \frac{\mu^{n+1} - \lambda^{n+1} - \hat\mu_0(\mu^n - \lambda^n)}{(\mu - \lambda)\lambda^n}, \quad 
\tilde x_n = \frac{\lambda}{\tilde s_n} = -\frac{\lambda}{\mu}\tilde y_n,\quad n\geq0.
$$
Hence, from \eqref{newbdcoe}, the birth--death rates of the Darboux-transformed process are
$$
\hat\lambda_n = \lambda\,\frac{\tilde s_{n+1}}{\tilde s_n},\quad 
\hat\mu_{n+1} = \mu\,\frac{\tilde s_{n}}{\tilde s_{n+1}}, \quad n \ge 0.
$$
For this transformed process, the parameter $\hat\mu_0$ can be chosen in the range $0 \le \hat\mu_0 \le \mu$.  
The associated spectral measure is $d\widehat{\psi}(x) = \frac{x}{\mu}d\psi(x)$, that is, the measure in \eqref{psiabs}. The corresponding orthogonal polynomials $(\hat Q_n)_{n \in \mathbb{N}_0}$ are given by
$$
\hat Q_n(x) = \frac{(\mu - \lambda)\lambda^n}{\mu^{n+1} - \lambda^{n+1} - \hat\mu_0(\mu^n - \lambda^n)}
\left(\frac{\mu}{\lambda}\right)^{n/2}
U_n\!\left(\frac{\lambda + \mu - x}{2\sqrt{\lambda\mu}}\right), \quad n \ge 0.
$$
Two particular choices of $\hat\mu_0$ lead to especially simple results:
\begin{itemize}[leftmargin=0.25in]
\item If $\hat\mu_0=\mu$, then
$\tilde s_n=1$, $\tilde x_n=\lambda$, and $\tilde y_n=-\mu$
for $n\geq0$, while $\tilde r_n=-1$ for $n\geq1$, so that $\hat\lambda_n = \lambda$ and $\hat\mu_n = \mu$ for all $n \ge 0$. The Darboux transformation is \emph{``almost'' invariant} in this situation: the original and transformed processes differ only in the absorption rate, since $\mu_0=\mu-\lambda$ and $\hat\mu_0=\mu$.
\item If $\hat\mu_0 = \lambda$ (which is possible since $\mu > \lambda$), then $\tilde s_n = \pi_n^{-1}$, $\tilde r_{n+1} = -\pi_n^{-1}$, $\tilde x_n = \lambda\pi_n$, and $\tilde y_n = -\mu\pi_n$ for $n\geq0$. In this case, $\hat\lambda_n = \mu$ and $\hat\mu_n = \lambda$ for all $n \ge 0$. The transformed rates have the same birth--death interchange as in the \emph{dual process} of Remark~\ref{remdual}. However, this is not an instance of that construction, since here the original process is nonconservative, with \(\mu_0=\mu-\lambda>0\).
\end{itemize}

\end{itemize}

\medskip

\hspace{-.35cm}\underline{\textbf{UL factorization}}. The coefficients $x_n, y_n, s_n,$ and $r_n$ in \eqref{xxnn2} and \eqref{yynn2}--\eqref{rrnn2} can be explicitly computed from \eqref{potcoefx1}--\eqref{pqm0ex}. Note that the formula for $\sum_{k=0}^n \pi_k u_k^2$, where $(u_n)_{n \in \mathbb{N}_0}$ is defined in \eqref{uun}, can be obtained from \eqref{pqm0ex} by replacing the parameter $\mu_0$ with $\mu_0 - x_0 - \tilde\mu_0$. The probabilistic interpretation of these coefficients was discussed in Section~\ref{picul}. In general, the resulting expressions are cumbersome, except for certain particular cases.

Recall that, in this context, we have to apply Theorem \ref{theom37}. On one hand, if $x_0+ \tilde\mu_0 \neq \mu_0 + (\lambda - \mu)\mathbf{1}_{\{\lambda > \mu\}}$, then $T=\infty$ and consequently
$\tilde{\mu}_0 = 0$. In this situation, there is a \emph{free parameter} $x_0$, which must satisfy
\[
0 < x_0 < \mu_0 + (\lambda - \mu)\,\mathbf{1}_{\{\lambda > \mu\}}.
\]
On the other hand, if $x_0+ \tilde\mu_0= \mu_0 + (\lambda - \mu)\mathbf{1}_{\{\lambda > \mu\}}$, the series $T=\sum_{n=0}^\infty \pi_n u_n^2$ in \eqref{TTT} (see also \eqref{TTT3}) can be studied from \eqref{pqm0ex} replacing $\mu_0$ by $\mu_0 - x_0 - \tilde\mu_0$, which in this case is given by $(\mu-\lambda)\mathbf{1}_{\{\lambda > \mu\}}$. A direct computation shows that
$$
T=\begin{cases}
\displaystyle \frac{\lambda}{\lambda - \mu},
& \mbox{if } \lambda > \mu\\[2.5ex]
\displaystyle \frac{\mu}{\mu-\lambda},
& \mbox{if } \lambda < \mu\\[2.5ex]
\infty, & \mbox{if } \lambda = \mu
\end{cases}.
$$
Then, for $\lambda\neq\mu$, the absorption rate $\tilde\mu_0$ must satisfy, following \eqref{condBAB},
\begin{equation}\label{rulespe}
0\leq\tilde\mu_0 \le \left(\mu_0 + (\lambda - \mu)\mathbf{1}_{\{\lambda > \mu\}}\right)\!\frac{|\lambda-\mu|}{\max\{\lambda,\mu\}},
\end{equation}
while, for $\lambda=\mu$, we have $T=\infty$. Consequently, $\tilde\mu_0=0$ and $x_0=\mu_0$, which is only possible when $\mu_0>0$.

Under these conditions, the Darboux transformation $\widetilde{\mathcal{A}}$ is the infinitesimal generator of a new birth--death process with birth and death rates $\{\tilde{\lambda}_n,\tilde{\mu}_n\}$ given by \eqref{newbdcoe2} and \eqref{laconcha3r}. By Proposition~\ref{propSUL}, this process again satisfies \eqref{uniqueness_assumption}. The corresponding spectral measure is given by the Geronimus transform of $d\psi$ in \eqref{spmec2e}. To compute it, we first determine the moment $m_{-1}$ of $d\psi$. From the Stieltjes transform \eqref{Bzpsiex} evaluated at $z=0$, or equivalently from \eqref{mmm1}, we obtain
\[
m_{-1}=\int_0^\infty x^{-1}\,d\psi(x)
=
\begin{cases}
\dfrac{1}{\lambda-\mu+\mu_0}, & \lambda>\mu,\\[1em]
\dfrac{1}{\mu_0}, & \lambda\leq\mu,\ \mu_0>0.
\end{cases}
\]
Notice that when $\mu_0=0$, the UL factorization is admissible only for $\lambda>\mu$, so the first expression applies. For $\mu_0>0$ and $\mu_0\neq\mu$, the only possible discrete mass away from the origin is located at $\zeta$ in \eqref{zzz} and occurs only when $\lambda\mu<(\mu-\mu_0)^2$. Hence, the transformed spectral measure is
\begin{equation}\label{SMEEx1DTUL}
\begin{split}
d\widetilde{\psi}(x)
&= 
\frac{(x_0 + \tilde\mu_0)\sqrt{4\lambda\mu - (\lambda + \mu - x)^2}}
{2\pi x[(\mu - \mu_0)x + \mu_0(\lambda - \mu + \mu_0)]}
\, \mathbf{1}_{\{x\in[\sigma_-,\sigma_+]\}}dx \\[0.5em]
&\qquad + 
\frac{(x_0 + \tilde\mu_0)(\lambda\mu - (\mu - \mu_0)^2)}
{\mu_0(\mu - \mu_0)(\lambda - \mu + \mu_0)}
\, \delta_\zeta(dx)\, \mathbf{1}_{\{\lambda\mu < (\mu - \mu_0)^2\}}+
\left(1 - \frac{x_0 + \tilde\mu_0}{\mu_0 + (\lambda - \mu)\mathbf{1}_{\{\lambda > \mu\}}}\right)
\delta_0(dx).
\end{split}
\end{equation}
In particular, when $\mu_0=\mu$, there is no discrete mass at $\zeta$, so only the absolutely continuous part and the possible mass at $x=0$ remain. The coefficient of $\delta_0$ is nonnegative by condition \eqref{cond1ULp}. Thus, depending on the parameters, $d\widetilde{\psi}$ may contain discrete masses at $\zeta$ and/or at $0$. The corresponding orthogonal polynomials can be obtained from \eqref{QDTUL} (see also \eqref{QQbb1}); however, their explicit form is generally difficult to express in closed form. We now turn to some particular cases where further simplifications occur.

\begin{itemize}[leftmargin=0.25in]

\item Case $\mu_0=\tilde\mu_0=0$. In this situation, $Q_n(0)=1$. This case is possible if and only if $A<\infty$, since otherwise condition \eqref{cond1ULp} would force $x_0=0$. Hence, from \eqref{quantABex}, we necessarily have $\lambda>\mu$, and the admissible range of the parameter is
\(
0<x_0\leq\lambda-\mu.
\)
For $0<x_0<\lambda-\mu$, we are in the first case of Theorem~\ref{theom37}, whereas the endpoint $x_0=\lambda-\mu$ corresponds to its second case. From \eqref{xxnn2} and \eqref{yynn2}--\eqref{rrnn2}, we obtain
\begin{align*}
x_n &= \frac{x_0 \mu^n}{\lambda^n - x_0 \left( \dfrac{\lambda^n - \mu^n}{\lambda - \mu} \right)}, 
\quad y_n = -x_n, \quad n \ge 0,\\[0.5em]
s_0 &= 1, \quad 
s_n = \frac{1}{x_0} \frac{\lambda^n}{\mu^{n-1}} 
- \frac{\lambda}{\mu^{n-1}} \left( \frac{\lambda^{n-1} - \mu^{n-1}}{\lambda - \mu} \right),\quad 
r_n = - \frac{1}{x_0} \frac{\lambda^n}{\mu^{n-1}} 
+ \frac{1}{\mu^{n-1}} \left( \frac{\lambda^n - \mu^n}{\lambda - \mu} \right), \quad n \ge 1.
\end{align*}
The birth--death rates of the Darboux-transformed process in \eqref{newbdcoe2} and \eqref{laconcha3r} are therefore
\begin{equation*}
\begin{split}
\tilde\lambda_n
&=
\frac{\lambda\mu\left[\lambda^{n-1}(\lambda-\mu)
-x_0(\lambda^{n-1}-\mu^{n-1})\right]}
{\lambda^n(\lambda-\mu)-x_0(\lambda^n-\mu^n)},
\quad n\geq1,
\qquad
\tilde\lambda_0=x_0,\\[0.5em]
\tilde\mu_{n+1}
&=
\frac{\lambda^{n+1}(\lambda-\mu)
-x_0(\lambda^{n+1}-\mu^{n+1})}
{\lambda^n(\lambda-\mu)-x_0(\lambda^n-\mu^n)},
\quad n\geq0,
\end{split}
\end{equation*}
and satisfy $\tilde\lambda_n\tilde\mu_n=\lambda\mu$ for all $n\geq1$. The corresponding spectral measure is
\begin{equation*}
d\widetilde{\psi}(x)
=
\frac{x_0\sqrt{4\lambda\mu-(\lambda+\mu-x)^2}}
{2\pi\mu x^2}
\mathbf{1}_{\{x\in[\sigma_-,\sigma_+]\}}dx
+
\left(1-\frac{x_0}{\lambda-\mu}\right)\delta_0(dx).
\end{equation*}
The associated orthogonal polynomials $(\tilde Q_n)_{n\in\mathbb{N}_0}$ can be obtained from \eqref{QDTUL}, although their explicit expressions do not simplify in general. At the endpoint $x_0=\lambda-\mu$, the discrete mass at $x=0$ vanishes and the factorization coefficients reduce to
\[
x_n = \lambda - \mu = -y_n, \quad n\geq0,\quad s_0 = 1, \quad 
s_n = \frac{\lambda}{\lambda - \mu}, \quad 
r_n = -\frac{\mu}{\lambda - \mu},\quad n\geq1.
\]
Consequently,
\[
\tilde\lambda_0 = \lambda - \mu, \quad 
\tilde\lambda_n = \lambda, \quad 
\tilde\mu_n = \mu, \quad n \ge 1.
\]
Thus, the Darboux transformation is \emph{almost invariant}, differing from the original process only in the rates at the boundary. In this endpoint case, the corresponding orthogonal polynomials admit an explicit expression. From \eqref{QQbb1}, for $n\geq1$,
\begin{align*}
\tilde Q_n(x)
&=
\frac{1}{\lambda-\mu}
\bigl(\lambda Q_n(x)-\mu Q_{n-1}(x)\bigr)\\[0.3em]
&=
\frac{1}{\lambda-\mu}
\left(\frac{\mu}{\lambda}\right)^{n/2}
\left[
\lambda U_n\!\left(\frac{\lambda+\mu-x}{2\sqrt{\lambda\mu}}\right)
-2\sqrt{\lambda\mu}\,
U_{n-1}\!\left(\frac{\lambda+\mu-x}{2\sqrt{\lambda\mu}}\right)
+\mu U_{n-2}\!\left(\frac{\lambda+\mu-x}{2\sqrt{\lambda\mu}}\right)
\right],
\end{align*}
while $\tilde Q_0(x)=1$. Using the three-term recurrence relation
\(U_{n+1}(y)=2yU_n(y)-U_{n-1}(y)\)
and the identity
\(2T_n(y)=U_n(y)-U_{n-2}(y)\),
with \(y=(\lambda+\mu-x)/(2\sqrt{\lambda\mu})\), we obtain, for $n \ge 1$,
\begin{equation*}
\tilde Q_n(x)
=
\frac{1}{\lambda-\mu}
\left(\frac{\mu}{\lambda}\right)^{n/2}
\left[
-2\mu T_n\!\left(\frac{\lambda+\mu-x}{2\sqrt{\lambda\mu}}\right)
+(\lambda+\mu)U_n\!\left(\frac{\lambda+\mu-x}{2\sqrt{\lambda\mu}}\right)
-2\sqrt{\lambda\mu}\,
U_{n-1}\!\left(\frac{\lambda+\mu-x}{2\sqrt{\lambda\mu}}\right)
\right].
\end{equation*}
These polynomials are known as \emph{perturbed Chebyshev polynomials}; see \cite[p.~204]{Ch}.

\item Let us now consider the case \(\mu_0>0\) in which absorption at \(-1\) is certain, that is, \(A=\infty\). From \eqref{quantABex}, this requires $\mu \ge \lambda$. Moreover, we focus on the situation where  
$T < \infty$. This condition holds when $\mu > \lambda$ and 
$x_0 = \mu_0 - \tilde\mu_0$, in which case the parameters $\mu_0$ and $\tilde\mu_0$ must satisfy (see \eqref{rulespe})
\[
0\leq\tilde\mu_0 \le \mu_0 \left( 1 - \frac{\lambda}{\mu} \right).
\]
For the special case $\tilde\mu_0 = \mu_0 \left( 1 - \frac{\lambda}{\mu} \right)$, we have $x_0 = \dfrac{\mu_0 \lambda}{\mu}$, and the coefficients $x_n, y_n, s_n,$ and $r_n$ are constant. Indeed,
\[
x_n = \frac{\mu_0 \lambda}{\mu}, \quad y_n = -\mu_0, \quad n \ge 0, 
\quad s_0 = 1, \quad s_n = \frac{\mu}{\mu_0}, \quad r_n = -\frac{\mu}{\mu_0}, \quad n \ge 1.
\]
The corresponding birth--death rates of the Darboux-transformed process, given by \eqref{newbdcoe2} and \eqref{laconcha3r}, are therefore
\[
\tilde\lambda_0 = \frac{\mu_0 \lambda}{\mu}, \quad 
\tilde\lambda_n = \lambda, \quad 
\tilde\mu_n = \mu, \quad n \ge 1.
\]
As before, the Darboux transformation is \emph{almost invariant}: the transformed process coincides with the original one except for the initial rates.  
Here, however, we have $\mu_0 > 0$, $\tilde\mu_0 = \mu_0 (1 - \lambda / \mu) > 0$, and $\tilde\lambda_0 = \mu_0 \lambda / \mu$ with $\mu > \lambda$. The corresponding spectral measure is given by \eqref{SMEEx1DTUL}.  
In this case, the measure has no discrete component at $x=0$, and a discrete mass at $x=\zeta$ appears only when $\mu_0>\mu+\sqrt{\lambda\mu}$ or $\mu_0<\mu-\sqrt{\lambda\mu}$.

\end{itemize}

\section{Concluding remarks}\label{secConc}

In this paper we have studied LU and UL factorizations of the infinitesimal generator of a possibly nonconservative birth--death process on $\mathbb{N}_0$, with the requirement that the corresponding Darboux transformations are again generators of birth--death processes. In contrast with the discrete-time setting of \cite{GdI}, this requirement is not automatic and imposes explicit restrictions on the parameters. These restrictions are characterized exactly in Theorems~\ref{theom33} and~\ref{theom37}: in the LU case the factorization is determined by the absorption rate $\hat\mu_0$ of the transformed process, which must lie in an explicit interval, while in the UL case there is an additional free parameter $x_0$, and the admissible region depends on whether $x_0+\tilde\mu_0$ attains its maximal value $\mu_0+1/A$. In both cases all the coefficients of the factors are given in closed form in terms of the birth--death rates, and they admit probabilistic interpretations in terms of hitting probabilities, conditional expected passage times and occupation times, related to Doob $h$-transforms and dual processes.

From the spectral point of view, the Darboux transformations correspond to a Christoffel transformation (LU case) and a Geronimus transformation (UL case) of the original spectral measure, so the Karlin--McGregor representation remains available for the new processes whenever their spectral measure is uniquely determined. We have shown that this is always the case for the UL factorization (Proposition~\ref{propSUL}), whereas for the LU factorization it is equivalent to the divergence of a single explicit series that depends only on the original rates (Proposition~\ref{propSLU}). Since these transformations act on the spectral measure near $x=0$, they can change the recurrence or absorption behavior of the process; the $M/M/1$ queue shows that a transient process may become positive recurrent after a Darboux transformation, and vice versa, and that some choices of parameters lead to Darboux transformations that are almost invariant, differing from the original process only at the boundary.

Several questions remain open. Darboux transformations do not preserve non-explosion in general, and it would be interesting to characterize when they do. Other natural directions are iterated factorizations, which lead to Christoffel and Geronimus transformations of higher order, and the extension of the present approach to continuous-time processes on other state spaces, such as bilateral birth--death processes on $\mathbb{Z}$ or quasi-birth--death processes, as done in discrete time in \cite{dIJ1,dIJ2,GdI4}. Finally, further explicit examples, such as the $M/M/\infty$ queue and linear birth--death processes, whose spectral measures are associated with the Charlier, Meixner and Laguerre polynomials, as well as birth--death processes with quadratic rates, will be studied in a forthcoming work.

\section*{Acknowledgments}
This work was partially supported by research project PID2024-155133NB-I00 funded by MICIU/AEI, Spain, and PIUAH25/CC-006, funded by the Universidad de Alcal\'a, Spain. The first author was supported by a doctoral scholarship from SENACYT, Panama, under program BIDP-2022-026.

\section*{Statement on the use of AI tools}
The authors used AI-based tools to assist in checking computations and consistency, and in editing the text and the \LaTeX{} source. All content was reviewed and verified by the authors, who take full responsibility for it.

\section*{Conflict of interest statement}
The authors declare no conflicts of interest. 

\section*{Data availability statement}

Data sharing is not applicable to this article as no new data were created or analyzed in this study.

\end{document}